\documentclass[11pt,reqno]{amsart}

\usepackage{amsfonts,amssymb,amsmath}

\usepackage[cp1251]{inputenc}
\usepackage[T2B]{fontenc}
\usepackage[english]{babel}
\usepackage[mathscr]{eucal}
\usepackage{calrsfs}

\usepackage{tikz}
\usepackage{ragged2e}

\usepackage{tikz}  
\usepackage[all]{xy}

\newtheorem{theorem}{Theorem}[section]

\newtheorem{proposition}[theorem]{Proposition}

\theoremstyle{remark}
\newtheorem{remark}[theorem]{Remark}
\newtheorem{example}[theorem]{Example}

\usepackage[left=2.0cm,right=2.2cm,
    top=2.5cm,bottom=2.5cm]{geometry}

\begin{document}

\centerline{\bf LOCALISATION PROPERTY OF BATTLE--LEMARI\'{E} WAVELETS' SUMS}\footnote{The research of the first author was supported by the Russian Science Foundation (project RSF-DST: 16-41-02004).}

\smallskip

\centerline{\textsc{Elena P. Ushakova$^{1,2}$, Kristina E. Ushakova$^{3}$}}  

\smallskip   

\centerline{\it $^{1}$Computing Center of the Far Eastern Branch of the Russian Academy of Sciences, Khabarovsk, Russia}
\centerline{\it $^{2}$Peoples' Friendship University of Russia, Moscow, Russia}
\centerline{\it $^{3}$Immanuel Kant Baltic Federal University, Institute of Living Systems, Kaliningrad, Russia}
\vskip 0.2cm
\centerline{\it E-mails: $^1$elenau@inbox.ru,  $^3$kristina.ushakova@outlook.com}


\vskip 0.2cm

\noindent\textit{Key words}: B--spline, spline wavelet, Battle--Lemari\'{e} scaling function and wavelet, construction, localisation, Nikolskii--Besov spaces.\\ \textit{MSC (2000)}: 42C40, 42B35.

\vskip 0.2cm
\textsc{Abstract:} {\footnotesize
Explicit formulae are given for a type of Battle--Lemari\'{e} scaling functions and related wavelets. Compactly supported sums of their translations are established and applied to alternative norm characterization of sequence spaces isometrically isomorphic to Nikolskii--Besov spaces on $\mathbb{R}$.}

\section{Introduction}

Battle--Lemari\'{e} scaling functions are polynomial splines with simple knots at $\mathbb{Z}$ obtained by orthogonalisation process of the B--splines. For $n\in\mathbb{N}$ the $n-$th order B--spline is defined recursively by
\begin{equation*}
B_n(x):=(B_{n-1}\ast B_0)(x)=\int_0^1 B_{n-1}(x-t)\,dt=\frac{x}{n}B_{n-1}(x)
+\frac{n+1-x}{n}B_{n-1}(x-1)
\end{equation*}  with $B_0=\chi_{[0,1)}$. It is known that $B_n$ generates multiresolution analysis of $L^2(\mathbb{R})$ \cite{Chui,W}. Moreover \cite{Chui}, 
\begin{itemize}
\item[a)] ${\rm supp}\,B_n=[0,n+1]$, $n\in\mathbb{N}\cup\{0\}$ and $B_n(x)>0$ for all $x\in(0,n+1)$, and $B_n\in C^{n-1}$ for $n\ge 1$;
\item[b)] the restriction of $B_n$ to each $[m,m+1]$, $m=0,\ldots,n$, is a polynomial of degree $n$;
\item[c)] the function $B_n(x)$ is symmetrical about $x=(n+1)/2$, that is $B_n(\frac{n+1}{2}-x)=B_n(\frac{n+1}{2}+x)$.\end{itemize}
Battle--Lemari\'{e} scaling functions and related wavelets play an important role in approximation theory, numerical analysis (see e.g. \cite{Me}), image, data and signal processing involving analysis of biological sequences and molecular biology--related signals, etc. \cite{Bio}. On the strength of the differentiation property
\begin{equation}\label{diff}
B'_n(x)=B_{n-1}(x)-B_{n-1}(x-1)\qquad\textrm{for a.a. }x\in\mathbb{R},\qquad n\in\mathbb{N},
\end{equation}
this function class has appeared to be an effective tool for solving problems related to the theory of integration and differentiation operators in function spaces \cite{NU}. There is a number of papers devoted to Battle--Lemari\'{e} scaling functions and wavelets. Most of them deal with their implicit or approximate expressions. An idea of how to find explicit formulae for this function class was given in works by I.Ya. Novikov and S.B. Stechkin (\cite[\S~7]{NS} and \cite[\S~15]{NS1}, see also \cite{NPS} and \cite{NPS1}). The problem, we dealt in \cite{NU}, concerned the operators' compactness and approximation properties. The solution method required, first of all, explicit formulae for the chosen wavelet system. The other important points were in finding their proper transformations and sums in order to localise non--compactly supported scaling functions and wavelets of this type and, making use of \eqref{diff}, connect their components (splines) with splines of lower or higher orders. All these questions were covered in \cite[\S~2.2.2, Lemma 5, Proposition 7, Corollary 8]{NU} for the scaling function and wavelets of the first order only. The results of the present work make us able to continue the study of compactness and approximation properties of integration operators in spaces of functions of higher smoothness than established in \cite{NU}. In the first part of this work we derive exact expressions for a family of Battle--Lemari\'{e} scaling functions and wavelets of all positive integer orders (\S~2). The second part of the paper is devoted to the "localisation property" of the class (\S~3). Namely, we establish compactly supported combinations $\Psi_n$ of  wavelets of Battle--Lemari\'{e} type (see \eqref{psi_fint} in Theorem \ref{PhiPsi}), which contribute to simple connection between two B--splines of different orders by an integration (differential) operator as well as to relation between dilations of $B-$splines (see \eqref{dym1}, \eqref{dym2} and \eqref{dymm}). Similar result is given for the scaling function (see \eqref{phi_fint} in Theorem \ref{PhiPsi} in combination with \eqref{diff}) and both are applied to equivalent norm characteristics in Besov type spaces (Proposition \ref{prop}).  

Throughout the paper, 
we use $\mathbb Z$, $\mathbb N$ and $\mathbb R$ for integers, natural and real numbers, respectively; symbol $\mathbb{C}$ for the complex plane.
We make use of marks $:=$ and
$=:$ for introducing new quantities. 

\section{Construction of Battle--Lemari\'{e} scaling functions and related wavelets} The Fourier transform of the normalised $B$--spline of the $n-$th order has the form
\begin{equation}\label{maska}\hat{B}_n(\omega)=[\hat{B}_0(\omega)]^{n+1}=\mathrm{e}^{-i(n+1)\omega/2}\left(\frac{\sin(\omega/2)}{\omega/2}\right)^{n+1}=\mathrm{e}^{-i(n+1)\omega/4}(\cos(\omega/4))^{n+1}\hat{B}_n(\omega/2).\end{equation} In particular, if $n=1$ then
\begin{equation}\label{B1(2x)}\hat{B}_1(\omega)=\frac{1}{2}\mathrm{e}^{-i\omega/2}\,[1+\cos(\omega/2)]\hat{B}_1(\omega/2)\end{equation} and, therefore,
\begin{equation}\label{N}
B_1(x)=\frac{1}{2}B_1(2x)+B_1(2x-1)+\frac{1}{2}B_1(2x-2).
\end{equation} The general two--scale relation formula for $B_n$, $n\in\mathbb{N}\cup\{0\}$, has the following form \cite[Chapter~4]{Chui}: \begin{equation}\label{two-scale}B_n(x)=2^{-n}\sum_{k=0}^{n+1}\frac{(n+1)!}{k!(n+1-k)!}B_n(2x-k).\end{equation} 

Fixed $n\in\mathbb{N}$ and any $d,\tau\in\mathbb{Z}$ put $B_{n;\,d,\tau}(x):=B_n(2^dx-\tau)$, $x\in\mathbb{R}$. Each $d\in\mathbb{Z}$ let $V_d$ denote the $L^2(\mathbb{R})-$closure of the linear span of the system $\{B_{n;\,d,\tau}\colon \tau\in\mathbb{Z}\}$. It is well--known \cite{Chui} that the spline spaces $V_d$, $d\in\mathbb{Z}$, which are generated by the {\it scaling function} $B_n$, constitute a multiresolution analysis of $L^2(\mathbb{R})$ in the sense that \begin{itemize}
\item[{\rm (i)}] $\ldots\subset V_{-1}\subset V_0\subset V_1\subset\ldots$;
\item[{\rm (ii)}] $\mathrm{clos}_{L^2(\mathbb{R})}\Bigl(\bigcup_{d\in\mathbb{Z}} V_d\Bigr)=L^2(\mathbb{R})$;
\item[{\rm (iii)}] $\bigcap_{d\in\mathbb{Z}} V_d=\{0\}$;
\item[{\rm (iv)}] for each $d$ the $\{B_{n;\,d,\tau}\colon \tau\in\mathbb{Z}\}$ is an unconditional (but not orthonormal) basis of $V_d$.\end{itemize} 
Further, there are the orthogonal complementary subspaces $\ldots, W_{-1},W_0,W_1,\ldots$ such that 
\begin{itemize}
\item[{\rm (v)}] $V_{d+1}=V_d\oplus W_d$ for all $d\in\mathbb{Z}$,
where $\oplus$ stands for $V_d\perp W_d$ and $V_{d+1}=V_d+W_d$. 
\end{itemize} Wavelet subspaces $W_d$, $d\in\mathbb{Z}$, related to the spline $B_n$, are also generated by some basic functions ({\it wavelets}) in the same manner as the spline spaces $V_d$, $d\in\mathbb{Z}$, are generated by the spline $B_n$. 

It is well--known that there exists another scaling function, whose integer  translations form an orthonormal system within the same multiresolution analysis. The function $\phi^{BL}_n\in L^2(\mathbb{R})$ satisfying
\begin{equation*}\label{BLphi}\hat{\phi}^{BL}_n(\omega)=\hat{B}_n(\omega)\left(\sum_{m\in\mathbb{Z}}\left|\hat{B}_n(\omega+2\pi m)\right|^2\right)^{-1/2}\end{equation*} is called the Battle--Lemari\'{e} scaling function \cite{B,B1,L}. Integer translations of ${\phi}^{BL}_n$ form an orthonormal basis in $V_0$ of the multiresolution analysis generated by $B_{n}$.

The $n-$th order Battle-Lemari\'{e} wavelet is the function $\psi^{BL}_n$ whose Fourier transform is
\begin{equation*}\label{BLpsi}\hat{\psi}^{BL}_n(\omega)=-\mathrm{e}^{-i\omega/2}\frac{\overline{\hat{\phi}^{BL}_n(\omega+2\pi)}}
{\overline{\hat{\phi}^{BL}_n(\omega/2+\pi)}}{\hat{\phi}^{BL}_n(\omega/2)}.\end{equation*} Integer translations of ${\psi}^{BL}_n$ form an orthonormal basis in $W_0$ of the multiresolution analysis generated by $B_{n}$. A scaling function $\phi$ and one of its associated wavelets $\psi$ form a wavelet system $\{\phi,\psi\}$. 

In order to derive explicit expressions for spline wavelet systems of all orders $n\in\mathbb{N}$ with properties analogous to $\phi^{BL}_n$ and $\psi^{BL}_n$ we follow the idea from \cite[\S~7]{NS} (see also \cite[\S~15]{NS1}, \cite{NPS} and \cite{NPS1}). 

\subsection{Auxiliary results}
Denote $$\mathbb{P}_n(\omega):=\sum_{m\in\mathbb{Z}}\left|\hat{B}_n(\omega+2\pi m)\right|^2$$ and write
\begin{equation}\label{phi_n}\hat{\phi}_n^\#(\omega)=\frac{\hat{B}_n(\omega)}{\sqrt{\mathbb{P}_n(\omega)}}=m_n^\#(\omega/2)\hat{\phi}^\#_n(\omega/2),\end{equation} where the low--pass filter $m_n^\#(\omega)$ (see \eqref{maska}) is given by $$m_n^\#(\omega):=
\mathrm{e}^{-i(n+1)\omega/2}(\cos(\omega/2))^{n+1}\sqrt{\frac{\mathbb{P}_n(\omega)}{\mathbb{P}_n(2\omega)}}.$$ A wavelet function ${\psi}_n^\#$ related to $\phi_n^\#$ must be of the form 
\begin{equation}\label{psi_n}\hat{\psi}_n^\#(\omega)={M}_n^\#(\omega/2)\hat{\phi}_n^\#(\omega/2),\end{equation}
where the high--pass filter ${M}_n^\#(\omega)$ is given by \begin{multline}\label{psi_n'}{M}_n^\#(\omega):=\mathrm{e}^{-i\omega}\overline{m_n^\#(w+\pi)}=\mathrm{e}^{-i\omega}\mathrm{e}^{i(n+1)(\omega+\pi)/2}(\sin(\omega/2))^{n+1}
\sqrt{\frac{\mathbb{P}_n(\omega+\pi)}{\mathbb{P}_n(2\omega)}}\\=\mathrm{e}^{-i\omega}\mathrm{e}^{i(n+1)\omega/2}(i\sin(\omega/2))^{n+1}
\sqrt{\frac{\mathbb{P}_n(\omega+\pi)}{\mathbb{P}_n(2\omega)}}=\mathrm{e}^{-i\omega}\left(\frac{\mathrm{e}^{i\omega}-1}{2}\right)^{n+1}
\sqrt{\frac{\mathbb{P}_n(\omega+\pi)}{\mathbb{P}_n(2\omega)}}.\end{multline} 

Consider $\mathbb{P}_n(\omega)$, $n\in\mathbb{N}$. It is well--known that $\mathbb{P}_0(\omega)=1$ and $\phi_0^\#$, $\psi_0^\#$ form the Haar system. Write
\begin{equation}\label{Phi}\mathbb{P}_n(\omega)=(2\sin(\omega/2))^{2(n+1)}\sum_{k\in\mathbb{Z}}\frac{1}{(\omega+2\pi k)^{2(n+1)}}=:\rho_{2(n+1)}(\omega).\end{equation} The expressions \eqref{Phi} were investigated in terms of rational polynomials in \cite{Scho}. Let us shortly explain milestones of this well--known theory (see also \cite[\S~6.4]{Chui}).
Differentiation of \eqref{Phi} yields 
$$\rho_{m+1}(\omega)=\cos(\omega/2)\rho_m(\omega)-\frac{2}{m}\sin(\omega/2)\rho'_m(\omega),\qquad m-\textrm{even}.$$ As it was mentioned before, $\rho_2(\omega)=1$. Following \cite[Part~B, \S~1.1]{Scho}, we express $\rho_m(\omega)$ as polynomials in the variable $z=\cos(\omega/2)$ by means of $\rho_m(\omega)=U_{m-2}(\cos(\omega/2)$ to obtain the recurrence relation for \begin{equation}\label{doshlo}U_{m+1}(z)=
z\,U_m(z)+\frac{1}{m+2}(1-z^2)U'_m(z)\end{equation} with $U_0(z)=1$. Being an even polynomial for even $m$, the $U_{m-2}(z)$ may be expressed as a polynomial $U^\ast_\kappa(y)$ in the variable $y=1-z^2$ of degree $\kappa=(m-2)/2$. By \cite[Part~B, Lemma~2]{Scho}, $U_{m-2}(z)$ has all simple and purely imaginary zeros $\{z_1,\ldots,z_{m-2}\}$. The change of variable transforms them into the zeros $\{\alpha_1,\ldots,\alpha_\kappa\}$ of $U^\ast_\kappa(y)$, which must all be positive and not less than $1$. Thus, since $U^\ast_\kappa(0)=U_{m-2}(1)=1$, $$\rho_{2(n+1)}(\omega)=U_{2n}(x)=\left(1-\frac{y}{\alpha_1}\right)\ldots \left(1-\frac{y}{\alpha_n}\right).$$ We may and shall assume that  $y>1$ (see \cite[Part~B, Lemma~3]{Scho}).
To find zeros of $1-y/\alpha_j$, $j=1,\ldots,n$, with respect to $\mathrm{e}^{\pm i\omega}$ in this case, we write $\alpha_j-y=\alpha_j-\sin^2(\omega/2)=0$, that is
$$2(2\alpha_j-1)+\mathrm{e}^{i\omega}+\mathrm{e}^{-i\omega}=\mathrm{e}^{\pm i\omega}\left(\mathrm{e}^{\mp 2i\omega}+2(2\alpha_j-1)\mathrm{e}^{\mp i\omega}+1\right)=0.$$ Thus, $\mathrm{e}^{\mp i\omega}=-(2\alpha_j-1)\pm2\sqrt{\alpha_j(\alpha_j-1)}$, where $\alpha_j>1$ and, therefore, $-(2\alpha_j-1)\pm2\sqrt{\alpha_j(\alpha_j-1)}\in\mathbb{R}$. Notice that
\begin{equation}\label{doshlo'}r_j:=(2\alpha_j-1)-2\sqrt{\alpha_j(\alpha_j-1)}\in(0,1)\end{equation} and, therefore, $1/r_j:=(2\alpha_j-1)+2\sqrt{\alpha_j(\alpha_j-1)}>1$. Denote $t_j=r_j^{\pm 1}$. Since $t_j=\bar{t_j}$, then
\begin{multline*}1-\frac{\sin^2(\omega/2)}{\alpha_j}=\frac{\mathrm{e}^{\pm i\omega}}{4\alpha_j}(\mathrm{e}^{\mp i\omega}+r_j)(\mathrm{e}^{\mp i\omega}+1/r_j)=
\frac{\mathrm{e}^{\pm i\omega}}{4\alpha_j\,t_j}(\mathrm{e}^{\mp i\omega}+t_j)(\mathrm{e}^{\mp i\omega}t_j+1)\\=
\frac{1}{4\alpha_j\,t_j}(\mathrm{e}^{\pm i\omega}t_j+1)(\mathrm{e}^{\mp i\omega}t_j+1)=
\frac{1}{4\alpha_j\,t_j}|\mathrm{e}^{\pm i\omega}t_j+1|^2.\end{multline*} 
Thus, with $t_j=r_j^{\pm 1}$, $j=1,\ldots,n$,
\begin{gather}\label{Phi1}\mathbb{P}_n(\omega)=\rho_{2(n+1)}(\omega)=\frac{1}{4\alpha_1\,t_1}|\mathrm{e}^{\pm i\omega}t_1+1|^2\ldots \frac{1}{4\alpha_n\,t_n}|\mathrm{e}^{\pm i\omega}t_n+1|^2.
\end{gather} In view of \eqref{phi_n}, \eqref{psi_n}, \eqref{psi_n'} and \eqref{Phi1} 
we construct a wavelet system of Battle--Lemari\'{e} type as follows.

\subsection{The construction} Denote $$\mathbb{A}_{t_j}^\pm(\omega):=
\frac{1}{2\sqrt{\alpha_j\,t_j}}(\mathrm{e}^{\pm i\omega}t_j+1),\qquad j=1,\ldots, n.$$ 
Describe the case $n=1$ before considering the general situation.
\begin{example} If $n=1$ then $\alpha_1=3/2$, $r_1=2-\sqrt{3}$ (see \eqref{doshlo} and \eqref{doshlo'} in \S~2.1) and we have \eqref{N} with 
\begin{equation}\label{B1def}B_1(x)=\begin{cases} x, & \quad\textrm{if}\quad 0\le x<1,\\ 2-x, & \quad\textrm{if}\quad 1\le x<2,\\ 0 & \quad\textrm{otherwise}.\end{cases}\end{equation} Taking into account \eqref{phi_n} and \eqref{Phi1} with $t_1=r_1$, define 
\begin{equation}\label{phiSt}\hat{\phi}_1(\omega)=\frac{\hat{B}_1(\omega)}{\mathbb{A}_{r_1}^\pm(\omega)}={2\sqrt{\alpha_1r_1}}\ \frac{\hat{B}_1(\omega)}{\mathrm{e}^{\pm i\omega}r_1+1}=:\hat{\phi}_1^\pm(\omega).\end{equation} If we choose $t_1=1/r_1$ in \eqref{Phi1} then
\begin{equation}\label{phiSt'}\hat{\varphi}_1(\omega)=2\sqrt{\alpha_1/r_1}\ \frac{\hat{B}_1(\omega)}{\mathrm{e}^{\pm i\omega}/r_1+1}=\frac{2\sqrt{\alpha_1/r_1}}{\mathrm{e}^{\pm i\omega}/r_1}\ \frac{\hat{B}_1(\omega)}{\mathrm{e}^{\mp i\omega}r_1+1}=2\sqrt{\alpha_1r_1}\ \frac{\hat{B}_1(\omega)}{\mathrm{e}^{\mp i\omega}r_1+1}\,\mathrm{e}^{\mp i\omega}=:\hat{\varphi}_1^\pm(\omega),\end{equation} which is a shifted version of $\hat{\phi}_1^\mp(\omega)$ of the form \eqref{phiSt}. 
Further, since for $0<r<1$ \begin{equation}\label{ryadok}\frac{1}{\mathrm{e}^{\pm i\omega}r+1}=\sum_{l=0}^\infty \left(-r\,\mathrm{e}^{\pm i\omega}\right)^l,\end{equation} then we can write \begin{equation}\label{Stphi}\phi_1^{\pm}(x)=2\sqrt{\alpha_1r_1}\sum_{l\ge 0}(-r_1)^{l}B_1(x\pm l)\ \ \textrm{and}\ \ \varphi_1^{\pm}(x)=2\sqrt{\alpha_1r_1}\sum_{l\ge 0}(-r_1)^{l}B_1(x\mp l\mp 1).\end{equation} 
Choosing between ${\pm}$ 
we obtain $\phi_1^+$ and $\varphi_1^-$ with "left tails" of the forms \begin{equation}\label{phi1-}
\phi_1^+(x)=2\sqrt{\alpha_1r_1}\sum_{l\le 0}(-r_1)^{-l} B_1(x-l)\ \ \textrm{and}\ \ 
\varphi_1^-(x)=2\sqrt{\alpha_1r_1}\sum_{l\le 0}(-r_1)^{-l} B_1(x-l+1)=\phi_1^+(x+1)\end{equation} and, analogously, $\phi_1^-$ and $\varphi_1^+$ with "right tails" of the forms \begin{equation}\label{phi1+}
\phi_1^-(x)=2\sqrt{\alpha_1r_1}\sum_{l\ge 0}(-r_1)^l B_1(x-l)\ \ \textrm{and}\ \ 
\varphi_1^+(x)=2\sqrt{\alpha_1r_1}\sum_{l\ge 0}(-r_1)^l B_1(x-l-1)=\phi_1^-(x-1).\end{equation} $\phi_1^\pm$ and $\varphi_1^\mp$ differ by integer shifts only and constitute the same multiresolution analysis of $L^2(\mathbb{R})$.

\begin{figure}
\includegraphics[scale=0.65]{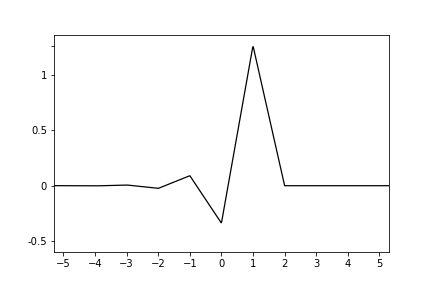}\caption{$\phi_1^+$}
\end{figure}
\begin{figure}
\includegraphics[scale=0.65]{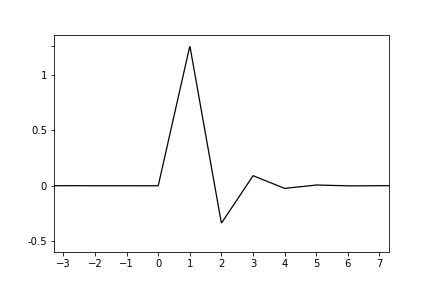}\caption{$\phi_1^-$}
\end{figure}
\begin{figure}
\includegraphics[scale=0.65]{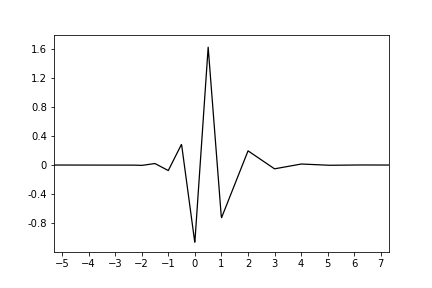}\caption{$\psi_{r_1}^+$}
\end{figure}
\begin{figure}
\includegraphics[scale=0.65]{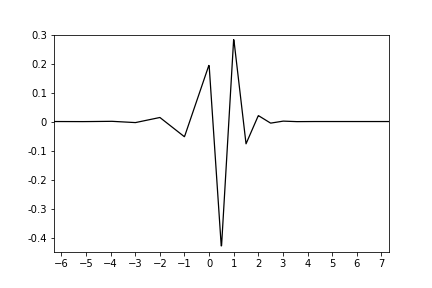}\caption{$r_1\psi_{1/r_1}^+$}
\end{figure}
\begin{figure}
\includegraphics[scale=0.65]{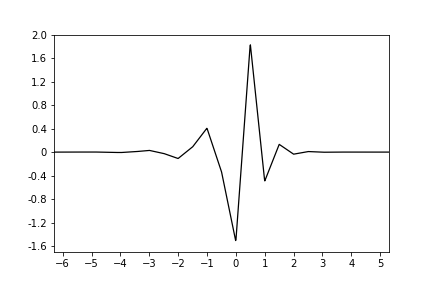}\caption{$\psi_{r_1}^-$}
\end{figure}
\begin{figure}
\includegraphics[scale=0.65]{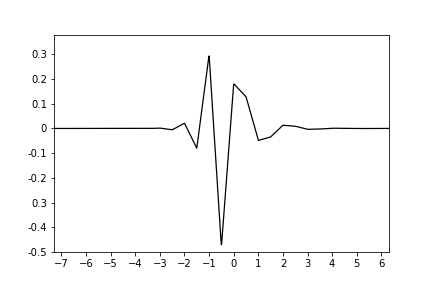}\caption{$r_1\psi_{1/r_1}^-$}
\end{figure}

To construct wavelets $\psi_1$ related to $\phi_1=\phi_1^\pm$ and $\phi_1=\varphi_1^\pm$ we write according to \eqref{psi_n}, \eqref{psi_n'} and \eqref{Phi1}:
$$\hat{\psi}_1(\omega)={M}_1(\omega/2)\hat{\phi}_1(\omega/2),$$ where (see also \eqref{B1(2x)}) $${M}_1(\omega):=\mathrm{e}^{-i\omega}\overline{{H}_1(\omega+\pi)},\qquad {H}_1(\omega)=\frac{\hat{\phi}_1(2\omega)}{\hat{\phi}_1(\omega)},\qquad \frac{\hat{B}_1(2\omega)}{\hat{B}_1(\omega)}=\mathrm{e}^{-i\omega}\cos^2(\omega/2),$$ that is $$H_1(\omega)=\cos^2(\omega/2){\frac{\mathrm{e}^{\pm i\omega}t_1+1}{\mathrm{e}^{\pm 2i\omega}t_1+1}}\mathrm{e}^{- i\omega},\qquad t_1=r_1^{\pm 1}.$$ Thus, taking into account that for any $r$
\begin{equation}\label{e-perehod}\frac{1}{[\mathrm{e}^{\mp i\omega}/r+1] [\mathrm{e}^{\pm i\omega/2}/r+1]}=
\frac{r^2\,\mathrm{e}^{\pm i\omega/2}}{[\mathrm{e}^{\pm i\omega}r+1] [\mathrm{e}^{\mp i\omega/2}r+1]},\end{equation}
we obtain
\begin{multline}\label{psiSt}\hat{\psi}_1(\omega)=-\frac{\sqrt{\alpha_1t_1}}{2}\ \frac{[1-t_1\mathrm{e}^{\mp i\omega/2}]\,[\mathrm{e}^{i\omega/2}-2+\mathrm{e}^{-i\omega/2}]}{[\mathrm{e}^{\mp i\omega}t_1+1] [\mathrm{e}^{\pm i\omega/2}t_1+1]}\hat{B}_1(\omega/2)\\=-\frac{r_1\sqrt{\alpha_1t_1}}{2}\begin{cases}\displaystyle \frac{[1/r_1-\mathrm{e}^{\mp i\omega/2}]\,[\mathrm{e}^{i\omega/2}-2+\mathrm{e}^{-i\omega/2}]}{[\mathrm{e}^{\mp i\omega}r_1+1] [\mathrm{e}^{\pm i\omega/2}r_1+1]}\hat{B}_1(\omega/2), & t_1=r_1\\
\displaystyle\frac{[r_1-\mathrm{e}^{\mp i\omega/2}]\,[\mathrm{e}^{i\omega/2}-2+\mathrm{e}^{-i\omega/2}]}{[\mathrm{e}^{\pm i\omega}r_1+1] [\mathrm{e}^{\mp i\omega/2}r_1+1]}\ \mathrm{e}^{\pm i\omega/2}\ \hat{B}_1(\omega/2), &t_1=1/r_1\end{cases},\end{multline} which means that we can construct two types of $\psi_1$: one having $[1/r_1-\mathrm{e}^{\mp i\omega/2}]$ in the numerator in $\hat{\psi}_1$ and the other one with $[r_1-\mathrm{e}^{\mp i\omega/2}]$. 
Symbol $\psi_{t_1}$ will be used to emphasise our choice of $t_1=r_1^{\pm 1}$ in \eqref{psiSt}.
 
With respect to \eqref{ryadok} and in view of $$[1/t_1-\mathrm{e}^{\mp i\omega/2}]\,[\mathrm{e}^{i\omega/2}-2+\mathrm{e}^{-i\omega/2}]=\mathrm{e}^{i\omega/2}/t_1-(1+2/t_1)+(2+1/t_1)\mathrm{e}^{\mp i\omega/2}-\mathrm{e}^{\mp i\omega}$$ the ${\psi}_{{1}}={\psi}_{{r_1}}^{\pm}$ with $\hat{\psi}_1$ defined by \eqref{psiSt} for $t_1=r_1$, have the following forms:\begin{multline}\label{Stpsi}{\psi}_{r_1}^{\pm}(x)=
-{r_1\sqrt{r_1\alpha_1}}\sum_{k\ge 0}(-r_1)^{k}\sum_{m\ge 0}(-r_1)^{m}\bigl[\frac{1}{r_1}B_1(2x\mp 2k\pm m+1)\\
-(1+\frac{2}{r_1})B_1(2x\mp 2k\pm m)+(2+\frac{1}{r_1})B_1(2x\mp 2k\pm m\mp 1)-B_1(2x\mp 2k \pm m\mp 2)\bigr].\end{multline} Analogously, for the ${\psi}_{{1}}={\psi}_{1/{r_1}}^{\pm}$ with $\hat{\psi}_1$ defined by \eqref{psiSt} for $t_1=1/r_1$, we get
\begin{multline}\label{StpsiD}{\psi}_{1/r_1}^{\pm}(x)=-
{\sqrt{\alpha_1r_1}}\sum_{k\ge 0}(-r_1)^{k}\sum_{m\ge 0}(-r_1)^{m}\bigl[{r_1}B_1(2(x\pm\frac{1}{2})\pm 2k\mp m+1)\\
-(1+{2}{r_1})B_1(2(x\pm\frac{1}{2})\pm 2k\mp m)+(2+{r_1})B_1(2(x\pm\frac{1}{2})\pm 2k\mp m\mp 1)-B_1(2(x\pm\frac{1}{2})\pm 2k\mp m\mp 2)\bigr].\end{multline} To obtain further results of the paper in the case $n=1$ (see  \S\S~3--4) we shall use wavelet systems $\{\phi_1^\pm,\,\psi_{r_1}^\pm\}$, $\{\phi_1^\pm,\,\psi_{1/r_1}^\pm\}$ and their shifted counterparts $\{\phi_1^\pm(\cdot\pm\frac{1}{2}),\,\psi_{r_1}^\pm(\cdot\pm\frac{1}{2})\}$ and $\{\phi_1^\pm(\cdot\mp\frac{1}{2}),\,\psi_{1/r_1}^\pm(\cdot\mp\frac{1}{2})\}$. \end{example}

For general $n\in\mathbb{N}$ we define $\hat{\phi}_n$ as follows:
\begin{equation}\label{'phi_n}\hat{\phi}_n(\omega)=\frac{\hat{B}_n(\omega)}{\mathbb{A}_{t_1}^\pm(w)\ldots \mathbb{A}_{t_n}^\pm(w)}=m_n(\omega/2)\hat{\phi}_n(\omega/2),\end{equation} where (see \eqref{maska}) $$m_n(\omega):=
\mathrm{e}^{-i(n+1)\omega/2}(\cos(\omega/2))^{n+1}\frac{\mathbb{A}_{t_1}^\pm(w)\ldots \mathbb{A}_{t_n}^\pm(w)}{\mathbb{A}_{t_1}^\pm(2w)\ldots \mathbb{A}_{t_n}^\pm(2w)}.$$ The Fourier transform of a wavelet function ${\psi}_n$ related to $\phi_n$ must satisfy the condition
\begin{equation}\label{'psi_n}\hat{\psi}_n(\omega)={M}_n(\omega/2)\hat{\phi}_n(\omega/2),\end{equation} where, with $\mathscr{A}_{t_j}^\mp(w)=\overline{\mathbb{A}_{t_j}^\pm(w+\pi)}=(1-\mathrm{e}^{\mp i\omega}t_j)/(2\sqrt{\alpha_j t_j})$, $j=1,\ldots,n$,
\begin{align}\label{'psi_n'}{M}_n(\omega):=\mathrm{e}^{-i\omega}\overline{m_n(w+\pi)}=&\mathrm{e}^{-i\omega}\mathrm{e}^{i(n+1)(\omega+\pi)/2}(\mathrm{e}^{i\pi}\sin(\omega/2))^{n+1}
\frac{\mathscr{A}_{t_1}^\mp(w)\ldots \mathscr{A}_{t_n}^\mp(w)}{\mathbb{A}_{t_1}^\mp(2w)\ldots \mathbb{A}_{t_n}^\mp(2w)}\nonumber\\=&\mathrm{e}^{-i\omega}\mathrm{e}^{i(n+1)\omega/2}(\mathrm{e}^{i\pi}i\sin(\omega/2))^{n+1}
\frac{\mathscr{A}_{t_1}^\mp(w)\ldots \mathscr{A}_{t_n}^\mp(w)}{\mathbb{A}_{t_1}^\mp(2w)\ldots \mathbb{A}_{t_n}^\mp(2w)}\nonumber\\
=&\mathrm{e}^{i\pi(n+1)}\mathrm{e}^{-i\omega}\left(\frac{\mathrm{e}^{i\omega}-1}{2}\right)^{n+1}
\frac{\mathscr{A}_{t_1}^\mp(w)\ldots \mathscr{A}_{t_n}^\mp(w)}{\mathbb{A}_{t_1}^\mp(2w)\ldots \mathbb{A}_{t_n}^\mp(2w)}.\end{align}
By \eqref{'phi_n}
\begin{equation}\label{phi_nn}
\hat{\phi}_n(\omega)=\frac{2^n\sqrt{\alpha_1\,t_1\ldots\alpha_n\,t_n}\ \hat{B}_n(\omega)}{(\mathrm{e}^{\pm i\omega}t_1+1)\ldots
(\mathrm{e}^{\pm i\omega}t_n+1)}.\end{equation}
Considering \eqref{Phi1}, we can vary $\pm$ in different terms in the denumerator of $\hat{\phi}_n$. But, for simplicity, in this work we limit our attention to the cases with either all "$+$" or all "$-$" only.

If $t_j=r_j$ for all $j=1,\ldots,n$, we write, taking into account \eqref{ryadok},
\begin{equation}\label{ex1}\hat{\phi}_n(\omega)=2^n\sqrt{\alpha_1\,r_1\ldots\alpha_n\,r_n}\sum_{l_1=0}^\infty (-r_1\,\mathrm{e}^{\pm i\omega})^{l_1}\ldots
\sum_{l_n=0}^\infty (-r_n\,\mathrm{e}^{\pm i\omega})^{l_n}\hat{B}_n(\omega),\end{equation} and, in view of $\mathscr{F}(B_n(\cdot-k))(\omega)= \mathrm{e}^{-ik\omega}\mathscr{F}(B_n(\cdot))(\omega)$, where $\mathscr{F}(f):=\hat{f}$, we obtain
\begin{equation}\label{phin}
\phi_n(x)=2^n\sqrt{\alpha_1\,r_1\ldots\alpha_n\,r_n}\sum_{l_1=0}^\infty (-r_1)^{l_1}\ldots
\sum_{l_n=0}^\infty (-r_n)^{l_n}B_n(x\pm l_1\ldots \pm l_n).\end{equation}
If $t_j=1/r_j$ for at least one of $j=1,\ldots,n$ we write, similarly to the case $n=1$:
\begin{equation}\label{buka}
\frac{1}{\mathbb{A}_{1/r_j}^\pm(\omega)}=\frac{2\sqrt{\alpha_j/r_j}}{\mathrm{e}^{\pm i\omega}/r_j+1}=
\frac{2\sqrt{\alpha_jr_j}\,\mathrm{e}^{\mp i\omega}}{\mathrm{e}^{\mp i\omega}r_j+1}=\frac{\mathrm{e}^{\mp i\omega}}{\mathbb{A}_{r_j}^\mp(\omega)}.\end{equation}
We say that $j\in J_r$ (or, alternatively, that $j\in J_{1/r}$), where $j\in\{1,\ldots,n\}$, if $t_j=r_j$ (or $t_j=1/r_j$). Let $c_{1/r}$ denote the cardinality of the set $J_{1/r}\subseteq \{1,\ldots,n\}$.
Then, it follows from \eqref{phi_nn} and \eqref{buka}, that
\begin{multline}\label{ex2}\hat{\phi}_n(\omega)=\frac{2^n\sqrt{\alpha_1\,r_1\ldots\alpha_n\,r_n}\ \hat{B}_n(\omega)\,
\mathrm{e}^{\mp c_{1/r}i\omega}}{\bigl[\prod_{j\in J_r} (\mathrm{e}^{\pm i\omega}r_j+1)\bigr]\bigl[\prod_{\iota\in J_{1/r}} (\mathrm{e}^{\mp i\omega}r_\iota+1)\bigr]}\\=
\beta_n \prod_{j\in J_r}\sum_{l_j=0}^\infty (-r_j\,\mathrm{e}^{\pm i\omega})^{l_j}\ 
\prod_{\iota\in J_{1/r}}\sum_{l_\iota=0}^\infty (-r_\iota\,\mathrm{e}^{\mp i\omega})^{l_\iota}\,\hat{B}_n(\omega)\ \mathrm{e}^{\mp c_{1/r}i\omega},\end{multline} where $\beta_n:=2^n\sqrt{\alpha_1\,r_1\ldots\alpha_n\,r_n}$. By definition of $\mathbb{P}_n(\omega)$ (see \S~2.1) and on the strength of \eqref{Phi1}, the system $\{\phi_n(\cdot -\tau)\}_{\tau\in\mathbb{Z}}$, defined by \eqref{ex1} or by \eqref{ex2}, is an orthonormal basis of $V_0$ generated by $\{B_{n;\,0,\tau}\colon \tau\in\mathbb{Z}\}$. 

In order to construct $\psi_n$, we write, taking into account \eqref{'psi_n}, \eqref{'psi_n'} and \eqref{'phi_n},
\begin{equation*}\label{obraz}\hat{\psi}_{n}(\omega)=\frac{(-1)^{n+1}}{2^{n+1}}\,\mathrm{e}^{-i\omega/2}\left({\mathrm{e}^{i\omega/2}-1}\right)^{n+1}
\frac{\mathscr{A}_{t_1}^\mp(w/2)\ldots \mathscr{A}_{t_n}^\mp(w/2)}{\mathbb{A}_{t_1}^\mp(w)\mathbb{A}_{t_1}^\pm(w/2)\ldots \mathbb{A}_{t_n}^\mp(w)\mathbb{A}_{t_n}^\pm(w/2)}\ \hat{B}_n(\omega/2),\end{equation*} where 
$$\mathrm{e}^{-i\omega/2}\left(\frac{\mathrm{e}^{i\omega/2}-1}{2}\right)^{n+1}=2^{-n-1}\sum_{k=0}^{n+1}\frac{(-1)^k(n+1)!}{k!(n+1-k)!}\mathrm{e}^{(n-k)i\omega/2}.$$ Then, on the strength of \eqref{ryadok} and \eqref{e-perehod},
\begin{multline}\label{psihat}\hat{\psi}_{n}(w)=\frac{\sqrt{\alpha_{1}\,t_{1}\ldots \alpha_{n}\,t
_{n}}}{2\cdot(-1)^{n+1}}(1-\mathrm{e}^{\mp i\omega/2}t_1)\ldots
(1-\mathrm{e}^{\mp i\omega/2}t_n)\ \mathrm{e}^{\pm c_{1/r}i\omega/2}\,\sum_{k=0}^{n+1}\frac{(-1)^k(n+1)!}{k!(n+1-k)!}\mathrm{e}^{(n-k)i\omega/2}\\
\times \prod_{j\in J_r}\sum_{m_j=0}^\infty (-r_j\,\mathrm{e}^{\mp i\omega})^{m_j}\sum_{l_j=0}^\infty (-r_j\,\mathrm{e}^{\pm i\omega/2})^{l_j}\  
\prod_{\iota\in J_{1/r}}r_\iota^2\sum_{m_\iota= 0}^\infty (-r_\iota\,\mathrm{e}^{\pm i\omega})^{m_\iota}\sum_{l_\iota= 0}^\infty (-r_\iota\,\mathrm{e}^{\mp i\omega/2})^{l_\iota}\ \hat{B}_n(\omega/2) \\
=\frac{\sqrt{\alpha_{1}\,t_{1}\ldots \alpha_{n}\,t
_{n}}(r_{1}\ldots r_{n})}{2\cdot(-1)^{n+1}}(1/t_1-\mathrm{e}^{\mp i\omega/2})\ldots
(1/t_n-\mathrm{e}^{\mp i\omega/2})\ \mathrm{e}^{\pm c_{1/r}i\omega/2}\,\sum_{k=0}^{n+1}\frac{(-1)^k(n+1)!}{k!(n+1-k)!}\mathrm{e}^{(n-k)i\omega/2}\\
\times \prod_{j\in J_r}\sum_{m_j=0}^\infty (-r_j\,\mathrm{e}^{\mp i\omega})^{m_j}\sum_{l_j=0}^\infty (-r_j\,\mathrm{e}^{\pm i\omega/2})^{l_j}\  
\prod_{\iota\in J_{1/r}}\sum_{m_\iota= 0}^\infty (-r_\iota\,\mathrm{e}^{\pm i\omega})^{m_\iota}\sum_{l_\iota= 0}^\infty (-r_\iota\,\mathrm{e}^{\mp i\omega/2})^{l_\iota}\ \hat{B}_n(\omega/2) .\end{multline} 
Denote ${\gamma }_{n}:=-{\sqrt{\alpha_{1}\,t_{1}\ldots \alpha_{n}\,t
_{n}}(r_1\ldots r_n)}$. Similarly to the case $n=1$, we shall use symbols $\psi_{t_1,\ldots,t_n}^\pm$ to emphasise the choice of $\mp$ and $t_j$, $j=1,\ldots,n$, in the product
$(1/t_1-\mathrm{e}^{\mp i\omega/2})\ldots
(1/t_n-\mathrm{e}^{\mp i\omega/2})$
in \eqref{psihat}.

Since $\mathscr{F}(B_n(2\cdot\pm k))(\omega)= \frac{1}{2}\mathrm{e}^{\pm ik\omega/2}\hat{B}_n(\omega/2)$, then the pre--image of
$$(1/t_1-\mathrm{e}^{\mp i\omega/2})\ldots
(1/t_n-\mathrm{e}^{\mp i\omega/2})\sum_{k=0}^{n+1}\frac{(-1)^k(n+1)!}{k!(n+1-k)!}\mathrm{e}^{(n-k)i\omega/2}\ \hat{B}_n(\omega/2)$$ is a sum of $2n+2$ translations of $\hat{B}_n(\omega/2)$ with coefficients depending on $n$ and $t_j$, $j=1,\ldots n$. It holds: 
\begin{multline}\label{ti}(1/t_1-\mathrm{e}^{\mp i\omega/2})\ldots
(1/t_n-\mathrm{e}^{\mp i\omega/2})=(-1)^n
\mathrm{e}^{\mp ni\omega/2}+\sum_{k=1}^n(-1)^{n-k}\sum_{\substack{j_1,\ldots,j_k\in\{1,\ldots,n\}\\ j_1\not=\ldots\not= j_k}}1/t_{j_1}\ldots 1/t_{j_k}\,\mathrm{e}^{\mp (n-k)i\omega/2}.\end{multline} 
The part \begin{equation}\label{ess}2^{-n-1}\hat{B}_n(\omega/2)\sum_{k=0}^{n+1}\frac{(-1)^k(n+1)!}{k!(n+1-k)!}\mathrm{e}^{(n-k)i\omega/2}\end{equation} in \eqref{psihat} is the most essential in the definition of $\hat{\psi}_{t_1,\ldots,t_n}^\pm$. Its similarity with the right hand side of the two--scale relation formula \eqref{two-scale} plays an important role in connection between dilations of $B-$splines (see \eqref{dym1}, \eqref{dym2} and \eqref{dymm}). Pre--image of \eqref{ess} has the form
$$2^{-n}\sum_{k=0}^{n+1}\frac{(-1)^k(n+1)!}{k!(n+1-k)!}\ {B}_n(2\cdot +(n-k)).$$ Moreover, up to a constant, it is equal to $(n+1)-$th order derivative of $B_{2n+1}(x+n)$ (see \eqref{diff}):
\begin{equation}\label{ddiff}2^{-n}\sum_{k=0}^{n+1}\frac{(-1)^k(n+1)!}{k!(n+1-k)!}\ {B}_n(2\cdot +(n-k))=2^{-2n-1}B_{2n+1}^{(n+1)}(2x+n).\end{equation}
Below is the expression for $\psi_{r_1,\ldots,r_n}^\pm$ with $\hat{\psi}_{r_1,\ldots,r_n}^\pm$ of the form \eqref{psihat} is given with respect to \eqref{ti}: 
\begin{multline}\label{even}{\gamma_n}^{-1}
\psi_{r_1,\ldots,r_n}^\pm(\cdot)=\sum_{m_1=0}^\infty(-r_1)^{m_1}\sum_{l_1= 0}^\infty (-r_1)^{l_1}\ldots \sum_{m_n= 0}^\infty(-r_n)^{m_n}\sum_{l_n= 0}^\infty(-r_n)^{l_n}\\\Biggl[\sum_{k=0}^{n+1}\frac{(-1)^{k}(n+1)!}{k!(n+1-k)!}\ 
B_n(2\cdot+(n-k)\mp n\mp 2m_1\pm l_1-\ldots \mp 2m_n\pm l_n)\\
-\biggl\{\!\sum_{j_1\in\{1,\ldots,n\}}\frac{1}{r_{j_1}}\!\biggr\}
\sum_{k=0}^{n+1}\frac{(-1)^{k}(n+1)!}{k!(n+1-k)!}\ 
B_n(2\cdot +(n-k)\mp (n-1)\mp 2m_1\pm l_1-\ldots \mp 2m_n\pm l_n)+\\
\biggl\{\!\sum_{\substack{j_1,j_2\in\{1,\ldots,n\}\\ j_1\not=j_2}}\!\!\!\frac{1}{r_{j_1}r_{j_2}}\!\biggr\}
\sum_{k=0}^{n+1}\frac{(-1)^{k}(n+1)!}{k!(n+1-k)!}\ 
B_n(2\cdot +(n-k)\mp (n-2)\mp 2m_1\pm l_1-\ldots \mp 2m_n\pm l_n)+\ldots\\+\biggl\{\!\!\sum_{\substack{j_1,\ldots,j_c\in\{1,\ldots,n\}\\ j_1\not=\ldots\not=j_c}}\!\!\!\frac{(-1)^c}{r_{j_1}\ldots r_{j_c}}\!\biggr\}
\sum_{k=0}^{n+1}\frac{(-1)^{k}(n+1)!}{k!(n+1-k)!}\ 
B_n(2\cdot +(n-k)\mp (n-c)\mp 2m_1\pm l_1-\ldots \mp 2m_n\pm l_n)\\ +\ldots+\frac{(-1)^n}{r_{1}\ldots r_{n}}\ 
\sum_{k=0}^{n+1}\frac{(-1)^{k}(n+1)!}{k!(n+1-k)!}\ 
B_n(2\cdot +(n-k)\mp 2m_1\pm l_1-\ldots \mp 2m_n\pm l_n)\Biggr]
.\end{multline} 
Analogously one can construct $\psi_{t_1,\ldots,t_n}$ with all $2^n$ possible combinations of $t_j=r_j^{\pm 1}$, $j=1,\ldots,n$. For the sake of convenience it is reasonable to center them at $(\cdot 
\mp \frac{c_{1/r}}{2})$. 
To obtain further results of the paper (see  \S\S~3--4) we shall use wavelet systems
\begin{equation}\label{ws}\Bigl\{\phi_n^\pm(\cdot\mp \frac{c_{1/r}}{2}),\,\psi_{t_1,\ldots,t_n}^\pm(\cdot\mp \frac{c_{1/r}}{2})\Bigr\} \quad\textrm{and}\qquad \Bigl\{\phi_n^\pm(\cdot\mp\frac{c_{1/r}}{2}\pm\frac{1}{2}),\,\psi_{t_1,\ldots,t_n}^\pm(\cdot\mp\frac{c_{1/r}}{2}\pm\frac{1}{2})\Bigr\}\end{equation} with $t_j=r_j^{\pm 1}$, $j=1,\ldots,n$. By construction, translations of those $\phi_n^\pm$ and $\psi_n^\pm:=\psi_{t_1,\ldots,t_n}^\pm$ form orthonormal basis in subspaces $V_0$ and $W_0$ of $L^2(\mathbb{R})$ related to the multiresolution analysis generated by $B_n(\cdot)$ and $B_n(\cdot\pm\frac{1}{2})$. Moreover, since $0<r_j<1$, $j=1,\ldots,n$, then 
the system $$\Bigl\{2^{d/2}h_{d\tau}^n\colon d\in\mathbb{N}_{-1},\ \tau\in\mathbb{Z}\Bigr\},$$ where $\mathbb{N}_{-1}:=\mathbb{N}_0\cup\{-1\}$ and $\mathbb{N}_0:=\mathbb{N}\cup\{0\}$, consisting of \begin{equation}\label{system} h_{-1\tau}^n(x):=\sqrt{2}\phi_n^\pm(x-\tau)\qquad\textrm{and}\qquad h^n_{d\tau}(x):=\psi_n^\pm(2^dx-\tau),\end{equation} is an orthonormal basis in $L^2(\mathbb{R})$ with the following properties:
\begin{itemize}
\item[{\rm(i)}] $\phi_n^\pm$ and $\psi_n^\pm$ have classical continuous derivatives up to order $n-1$ inclusively on $\mathbb{R}$;
\item[{\rm(ii)}] the restriction of $\phi_n^\pm$ and $\psi_n^\pm$ to each interval $(k,k+\frac{1}{2})$ with $2k\in\mathbb{Z}$ is a polynomial of degree at most $n$;
\item[{\rm(iii)}] there are constants $c>0$ and $\alpha>0$ such that for $m=0,1,\ldots,n$
$$\left|\frac{d^m}{dx^m}\phi_n^\pm(x)\right|+\left|\frac{d^m}{dx^m}\psi_n^\pm(x)\right|\le c\cdot\mathrm{e}^{-\alpha|x|},\qquad 2x\in\mathbb{R}\setminus\mathbb{Z}; $$
\item[{\rm(iv)}] for $m=0,1,\ldots,n$ $$\int_\mathbb{R} x^m \,\psi_n^\pm(x)\,dx=0.$$
\end{itemize} The (iv) follows, in particular, from \cite[\S~3.7, Proposition 4]{Me} (see also \cite[\S~2.5]{Tr5}).

\smallskip We complete the section by example of $\phi_n^\pm$ and $\psi_n^\pm$ of order $n=2$.

\begin{example}
Let $n=2$. Then $U_4(x)=(2x^4+11x^2+2)/15=(2y^2-15y+15)/15=U^\ast_2(y)$ and 
$\alpha_{1,2}=(15\pm\sqrt{105})/4$. Therefore, $2r_{1}=(13-\sqrt{105})-\sqrt{270-26\sqrt{105}}$~and $2r_{2}=(13+\sqrt{105})-\sqrt{270+26\sqrt{105}}$. By \eqref{phin}, we have with $\beta_2=4\sqrt{\alpha_1\,\alpha_2\,r_1\,r_2}$ (four options, depending on $\pm$, in total):
\begin{equation}\label{phinn}
\phi_2^\pm(x)=\beta_2\sum_{l=0}^\infty (-r_1)^{l}
\sum_{m=0}^\infty (-r_2)^{m}B_2(x\pm l\pm m),\end{equation} where
\begin{equation}B_2(x)=\begin{cases}
\frac{1}{2}x^2, & 0\le x<1,\\
-x^2+3x-\frac{3}{2}, & 1\le x<2,\\
\frac{1}{2}x^2-3x+\frac{9}{2}, & 2\le x<3\\ 0 & \textrm{otherwise}.
\end{cases}\end{equation}
\begin{figure}
\includegraphics[scale=0.65]{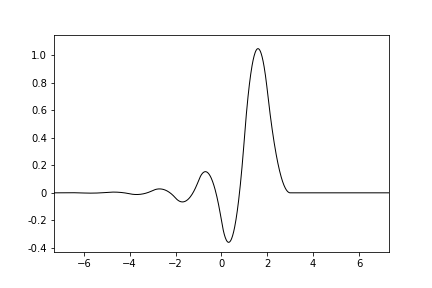}\caption{$\phi_2^+$}
\end{figure} \begin{figure}
\includegraphics[scale=0.65]{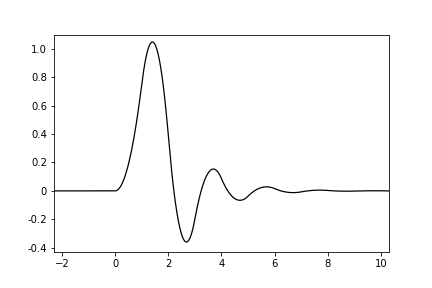}\caption{$\phi_2^-$}
\end{figure}
By \eqref{psihat} and \eqref{ryadok}, say, for $t_j=r_j$, $j=1,2$, (that is in the case $c_{1/r}=0$) \begin{multline*}\hat{\psi}_{r_1,r_2}^\pm(x)=\frac{\gamma_2}{2}(1/r_1-\mathrm{e}^{\mp i\omega/2})(1/r_2-\mathrm{e}^{\mp i\omega/2})
\left[\mathrm{e}^{i\omega}-3
\mathrm{e}^{i\omega/2}+3-\mathrm{e}^{-i\omega/2}\right]\ \hat{B}_2(\omega/2)\\
\times \sum_{k_1=0}^\infty (-r_1\,\mathrm{e}^{\pm i\omega/2})^{k_1}\sum_{k_2=0}^\infty (-r_2\,\mathrm{e}^{\pm i\omega/2})^{k_2}\sum_{m_1= 0}^\infty (-r_1\,\mathrm{e}^{\mp i\omega})^{m_1}\sum_{m_2\ge 0}^\infty (-r_2\,\mathrm{e}^{\mp i\omega})^{m_2}\end{multline*}
with $\gamma_2=-r_1\,r_2\sqrt{\alpha_1\,r_1\,\alpha_2\,r_2}$. Since \begin{multline*}(1/r_1-\mathrm{e}^{- i\omega/2})(1/r_2-\mathrm{e}^{- i\omega/2})
\left[\mathrm{e}^{i\omega}-3
\mathrm{e}^{i\omega/2}+3-\mathrm{e}^{-i\omega/2}\right]\\= \left[\mathrm{e}^{- i\omega}-(1/r_1+1/r_2)\mathrm{e}^{- i\omega/2}+1/(r_1r_2)\right]\left[\mathrm{e}^{i\omega}-3
\mathrm{e}^{i\omega/2}+3-\mathrm{e}^{-i\omega/2}\right]\\
=1/(r_1r_2)\mathrm{e}^{i\omega}-(1/r_1+1/r_2+3/(r_1r_2))\mathrm{e}^{i\omega/2} +(1+3(1/r_1+1/r_2)+3/(r_1r_2)) \\-(3+3(1/r_1+1/r_2)+1/(r_1r_2))\mathrm{e}^{-i\omega/2} +(3+1/r_1+1/r_2)\mathrm{e}^{-i\omega}-\mathrm{e}^{-3i\omega/2},\end{multline*}  \begin{figure}
\includegraphics[scale=0.65]{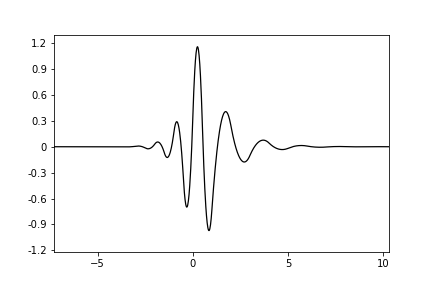}\caption{$\psi_{r_1,r_2}^+$}
\end{figure}
a wavelet $\psi_{r_1,r_2}^+$, related to $\phi_2^\pm$ may have, in particular, the following form:
\begin{multline}{\gamma_2}^{-1}\psi_{r_1,r_2}^+(x)=\\=\sum_{k_1\ge 0}(-r_1)^{k_1}
\sum_{k_2\ge 0}(-r_2)^{k_2}\sum_{m_1\ge 0}(-r_1)^{m_1}
\sum_{m_2\ge 0}(-r_2)^{m_2}\bigl[B_2(2x+k_1+k_2-2m_1-2m_2+2)/(r_1r_2)\\
-(1/r_1+1/r_2+3/(r_1r_2))B_2(2x+k_1+k_2-2m_1-2m_2+1)\\
+(1+3(1/r_1+1/r_2)+3/(r_1r_2))
B_2(2x+k_1+k_2-2m_1-2m_2)\\
 -(3+3(1/r_1+1/r_2)+1/(r_1r_2))B_2(2x+k_1+k_2-2m_1-2m_2-1)\\+(3+1/r_1+1/r_2)
 B_2(2x+k_1+k_2-2m_1-2m_2-2)-B_2(2x+k_1+k_2-2m_1-2m_2-3)
 \bigr].\end{multline}
\end{example}

\section{Localisation property} For a function $F$ on $\mathbb{R}$ and $2m\in\mathbb{N}$ put $$(S_1^{\pm m}F)(x):=F(x)+r_{j_1}\cdot F(x\pm m),\qquad (R_1^{\pm m}F)(x):=F(x)+r_{\iota_1}\cdot F(x\pm m)$$ and define, recursively,
\begin{gather*}(S_\epsilon^{\pm m}\ldots S_1^{\pm m}F)(x):=(S_{\epsilon-1}^{\pm m}\ldots S_1^{\pm m}F)(x)+r_{j_\epsilon}\cdot (S_{\epsilon-1}^{\pm m}\ldots S_1^{\pm m}F)(x\pm m),\qquad 1\le\epsilon\le c_r,\, {j_\epsilon}\in J_r;\\
(R_\upsilon^{\pm m}\ldots R_1^{\pm m}F)(x):=(R_{\upsilon-1}^{\pm m}\ldots R_1^{\pm m}F)(x)+r_{\iota_\upsilon}\cdot (R_{\upsilon-1}^{\pm m}\ldots R_1^{\pm m}F)(x\pm m),\quad 1\le\upsilon\le c_{1/r},\, {\iota_\upsilon}\in J_{1/r}.\end{gather*}
Let $\{\phi_n,\,\psi_n\}$ be wavelet systems \eqref{ws} with $\hat{\phi}_n=\hat{\phi}_n^\pm$ satisfying \eqref{ex2} and $\hat{\psi}_n=\hat{\psi}_{t_1,\ldots,t_n}^\pm$ of the form \eqref{psihat}. For simplicity we assume that $\phi_n$ are centred at $(\cdot \pm{c_{1/r}})$ and $\psi_n$ are centred at $(\cdot \mp\frac{c_{1/r}}{2})$.
~Recall that $c_{1/r}$ denotes cardinality of the set $J_{1/r}$. Similarly, $c_r$ stands for cardinality of the set $J_r$.
\begin{theorem}\label{PhiPsi} Fixed $n\in\mathbb{N}$, put $\beta_{n}:=2^{n}\sqrt{\alpha_{1}\,r
_{1}\ldots \alpha_{n}\,r_{n}}$, ${\gamma }_{n}:=-{\sqrt{\alpha_{1}\,t_{1}\ldots \alpha_{n}\,t
_{n}}(r_1\ldots r_n)}$, $\widetilde{\gamma }_{n}:=\frac{2^n\sqrt{\alpha_{1}\ldots \alpha_{n}}(r_1\ldots r_n)}{(-1)^{n+1}}$ and $\delta_n:=\prod_{j=1}^n(\frac{1}{r_j}-r_j)$.\\ It holds \begin{equation}\label{phi_fint}\Phi_n(x):=\Bigl(S_{c_r}^{\pm 1}\ldots S_1^{\pm 1}\bigl[R_{c_{1/r}}^{\mp 1}\ldots R_1^{\mp 1}\phi_{n}^\pm\bigr]\Bigr)(x)=\beta_n\,B_n(x)\end{equation} and
\begin{multline}
\label{psi_fint}\frac{\delta_{n}\widetilde{\gamma}_{n}}{2^{2n+1}}\ B_{2n+1}^{(n+1)}(2x+n)=\Psi
_{n}(x):
=\\=\frac{1}{\sqrt{r_n}}\Biggm(\ldots\frac{1}{\sqrt{r_4}}\Biggl[\frac{1}{\sqrt{r_3}}\biggl \{\frac{1}{\sqrt{r_2}}\Bigl(\frac{1}{\sqrt{r_1}}\Lambda^{\pm }_{{r_{1}},{r_{2}},{r_{3}},{r_{4}},
\ldots ,{r_{n}}}(x)-\sqrt{r_1}\Lambda^{\pm }_{\frac{1}{r_{1}},{r_{2}},{r_{3}},
{r_{4}},\ldots ,{r_{n}}}(x)\Bigr)\\-\sqrt{r_2}\Bigl(\frac{1}{\sqrt{r_1}}
\Lambda_{{r_{1}}, \frac{1}{r_{2}},{r_{3}},{r_{4}},\ldots ,{r_{n}}}
^{\pm }(x)-\sqrt{r_1}\Lambda^{\pm }_{\frac{1}{r_{1}},\frac{1}{r_{2}},{r_{3}},
{r_{4}},\ldots ,{r_{n}}}(x)\Bigr)\biggr \}
\\
-\sqrt{r_3}
\biggl \{\frac{1}{\sqrt{r_2}}\Bigl(\frac{1}{\sqrt{r_1}}
\Lambda_{{r_{1}}, {r_{2}},\frac{1}{r_{3}},{r_{4}},\ldots ,{r_{n}}}
^{\pm }(x)-\sqrt{r_1}
\Lambda^{\pm }_{\frac{1}{r_{1}},{r_{2}},\frac{1}{r_{3}},
{r_{4}},\ldots ,{r_{n}}}(x)\Bigr)\\-\sqrt{r_2}\Bigl(\frac{1}{\sqrt{r_1}}
\Lambda_{{r_{1}}, \frac{1}{r_{2}},\frac{1}{r_{3}},{r_{4}},\ldots ,
{r_{n}}}^{\pm }(x)-\sqrt{r_1}\Lambda^{\pm }_{\frac{1}{r_{1}},\frac{1}{r_{2}},\frac{1}{r
_{3}},{r_{4}},\ldots ,{r_{n}}}(x)\Bigr)\biggr \}\Biggr]
\\
-\sqrt{r_4}\Biggl[\frac{1}{\sqrt{r_3}}\biggl \{\frac{1}{\sqrt{r_2}}\Bigl(\frac{1}{\sqrt{r_1}}\Lambda^{\pm }_{{r_{1}},{r_{2}},{r_{3}},\frac{1}{r
_{4}},\ldots ,{r_{n}}}(x)-\sqrt{r_1}\Lambda^{\pm }_{\frac{1}{r_{1}},{r_{2}},
{r_{3}},\frac{1}{r_{4}},\ldots ,{r_{n}}}(x)\Bigr)\\-\sqrt{r_2}\Bigl(\frac{1}{\sqrt{r_1}}
\Lambda_{{r_{1}},\frac{1}{r_{2}},{r_{3}},\frac{1}{r_{4}},\ldots ,
{r_{n}}}^{\pm }(x)-\sqrt{r_1}\Lambda^{\pm }_{\frac{1}{r_{1}},\frac{1}{r_{2}},
{r_{3}},\frac{1}{r_{4}},\ldots ,{r_{n}}}(x)\Bigr)\biggr \}
\\
-\sqrt{r_3}
\biggl \{\frac{1}{\sqrt{r_2}}\Bigl(\frac{1}{\sqrt{r_1}}
\Lambda_{{r_{1}}, {r_{2}},\frac{1}{r_{3}},\frac{1}{r_{4}},\ldots ,
{r_{n}}}^{\pm }(x)-\sqrt{r_1}
\Lambda^{\pm }_{\frac{1}{r_{1}},{r_{2}},\frac{1}{r
_{3}},\frac{1}{r_{4}},\ldots ,{r_{n}}}(x)\Bigr)\\-\sqrt{r_2}\Bigl(\frac{1}{\sqrt{r_1}}
\Lambda_{{r_{1}}, \frac{1}{r_{2}},\frac{1}{r_{3}},\frac{1}{r_{4}},
\ldots ,{r_{n}}}^{\pm }(x)-\sqrt{r_1}\Lambda^{\pm }_{\frac{1}{r_{1}},\frac{1}{r
_{2}},\frac{1}{r_{3}},\frac{1}{r_{4}},\ldots ,{r_{n}}}(x)\Bigr)\biggr
\}\Biggr]\ldots\Biggm)-{\sqrt{r_n}}\Biggm(\ldots\Biggm) ,
\end{multline} where $B_{2n+1}^{(n+1)}(\cdot)$ stands for the $(n+1)-$th order derivative of $B_{2n+1}(\cdot)$ and
\begin{equation}\label{lambda}\Lambda_{{t_1},\ldots,{t_n}}^\pm(x):=
\Bigl(S_{c_r}^{\mp 1}S_{c_r}^{\pm 1/2}\ldots S_1^{\mp 1}S_1^{\pm 1/2}\bigl[R_{c_{1/r}}^{\pm 1}R_{c_{1/r}}^{\mp 1/2}\ldots R_1^{\pm 1}R_1^{\mp 1/2}\psi_{t_1,\ldots,t_n}^\pm\bigr]\Bigr)(x).\end{equation}
\end{theorem}

\begin{proof}
Since $\phi_n^\pm$ is centred at $(\cdot\pm c_{1/r})$ we obtain by \eqref{ex2}, 
\begin{multline}\label{combin}
\mathscr{F}\Bigl(S_{c_r}^{\pm 1}\ldots S_1^{\pm 1}\bigl[R_{c_{1/r}}^{\mp 1}\ldots R_1^{\mp 1}\phi_{n}^\pm\bigr]\Bigr)(\omega)\\=
{\mathscr{F}\Bigl(S_{c_r-1}^{\pm 1}\ldots S_1^{\pm 1}\bigl[R_{c_{1/r}}^{\mp 1}\ldots R_1^{\mp 1}\phi_{n}^\pm\bigr]\Bigr)(\omega)+r_{j_{c_r}}\mathrm{e}^{\pm i\omega}\cdot \mathscr{F}\Bigl(S_{c_r-1}^{\pm 1}\ldots S_1^{\pm 1}\bigl[R_{c_{1/r}}^{\mp 1}\ldots R_1^{\mp 1}\phi_{n}^\pm\bigr]\Bigr)(\omega)}\\=(1+r_{j_{c_r}}\mathrm{e}^{\pm i\omega})
{\mathscr{F}\Bigl(S_{c_r-1}^{\pm 1}\ldots S_1^{\pm 1}\bigl[R_{c_{1/r}}^{\mp 1}\ldots R_1^{\mp 1}\phi_{n}^\pm\bigr]\Bigr)(\omega)}=\ldots=
\bigl[\prod_{\epsilon=1}^{c_r} (\mathrm{e}^{\pm i\omega}r_{j_\epsilon}+1)\bigr]
{\mathscr{F}\Bigl(R_{c_{1/r}}^{\mp 1}\ldots R_1^{\mp 1}\phi_{n}^\pm\Bigr)(\omega)}\\=\bigl[\prod_{\epsilon=1}^{c_r} (\mathrm{e}^{\pm i\omega}r_{j_\epsilon}+1)\bigr]\biggl[
{\mathscr{F}\Bigl(R_{c_{1/r}-1}^{\mp 1}\ldots R_1^{\mp 1}\phi_{n}^\pm\Bigr)(\omega)}+r_{j_{c_{1/r}}}{\mathscr{F}\Bigl(R_{c_{1/r}-1}^{\mp 1}\ldots R_1^{\mp 1}\phi_{n}^\pm\Bigr)(\omega)}\biggr]\\=\bigl[\prod_{\epsilon=1}^{c_r} (\mathrm{e}^{\pm i\omega}r_{j_\epsilon}+1)\bigr](1+r_{j_{c_{1/r}}})\mathscr{F}\Bigl(R_{c_{1/r}-1}^{\mp 1}\ldots R_1^{\mp 1}\phi_{n}^\pm\Bigr)(\omega)\\\ldots=
\bigl[\prod_{\epsilon=1}^{c_r} (\mathrm{e}^{\pm i\omega}r_{j_\epsilon}+1)\bigr]\bigl[\prod_{\upsilon=1}^{c_{1/r}} (\mathrm{e}^{\mp i\omega}r_{\iota_\upsilon}+1)\bigr]
{\mathscr{F}(\phi_{n}^\pm)(\omega)}={\beta_{n}}\mathscr{F}(B_n)(\omega),\end{multline} and based on \eqref{phi_fint} is proven. 

On the strength of definition \eqref{psihat} of $\hat{\psi }^{\pm }
_{n}$, we obtain, similarly to \eqref{combin}, that
\begin{equation*}
\frac{2\cdot(-1)^{n}}{\gamma_{n}}\mathscr{F}\bigl(
\Lambda_{{t_{1}},\ldots ,{t_{n}}}^{\pm }\bigr)(\omega )= (1/t_{1}-
\mathrm{e}^{\mp i\omega /2})\ldots (1/t_{n}-\mathrm{e}^{\mp i\omega /2})
\ \sum_{k=0}^{n+1}\frac{(-1)^{k}(n+1)!}{k!(n+1-k)!}\mathrm{e}^{(n-k)i
\omega /2}\ \hat{B}_n(\omega/2),
\end{equation*}
where $\Lambda_{{t_{1}},\ldots ,{t_{n}}}^{\pm }$ defined by\eqref{lambda}. Recall that $t_{j}=r_{j}^{\pm 1}$, $j=1,\ldots ,n$.

Given $n\ge 1$ and chosen ``$+$'' or ``$-$'' in $
\Lambda_{{t_{1}},\ldots ,{t_{n}}}^{\pm }$ we consider the collection
$\mathbb{T}^{0}$ of $2^{n}$ functions $T^{0}_{{t_{1}},\ldots ,{t_{n}}}(x):=
\Lambda_{{t_{1}},\ldots ,{t_{n}}}^{\pm }(x)$, meaning $2^{n}$ different
combinations of $t_{j}=r^{\pm 1}_{j}$, $j=1,\ldots ,n$, in $
\Lambda_{{t_{1}},\ldots ,{t_{n}}}^{\pm }$. We pair elements of
$\mathbb{T}^{0}$ as follows: $T^{0}_{{t_{1}},\ldots ,{t_{n}}}$ and
$T^{0}_{{t'_{1}},\ldots ,{t'_{n}}}$ from $\mathbb{T}^{0}$ are coupled
if ${t_{j}}=t'_{j}$ for $j=2,\ldots ,n$, while ${t_{1}}=1/t'_{1}$. Each
couple we associate with the function
\begin{equation*}
T^{1}_{t_{2},\ldots , t_{n}}=\frac{1}{\sqrt{r_1}}T^{0}_{{r_{1}},t_{2},\ldots ,t_{n}}-\sqrt{r_1}T
^{0}_{\frac{1}{r_{1}},t_{2},\ldots ,t_{n}}
\end{equation*}
and call $\mathbb{T}^{1}$ the new collection of $2^{n-1}$ functions
$T^{1}_{t_{2},\ldots ,t_{n}}$. If $n=1$, we finish the localisation
process and obtain the function $T^{1}=\frac{1}{\sqrt{r_1}}\Lambda_{r_{1}}^{\pm }-\sqrt{r_1}
\Lambda_{\frac{1}{r_{1}}}^{\pm }$ satisfying
\begin{equation*}
\frac{\hat{T}^{1}(\omega )}{\widetilde{\gamma}_{1}(1/r_{1}-r_{1})}
=\frac{\hat{B}
_{1}(\omega /2)}{4}\sum_{k=0}^{2}\frac{(-1)^{k}2!}{k!(2-k)!}\mathrm{e}
^{(1-k)i\omega /2}=
\frac{\hat{B}_{1}(\omega /2)}{2}\left( \frac{
\mathrm{e}^{i\omega /2}}{2}-1+\frac{ \mathrm{e}^{-i\omega /2}}{2}\right) =:\frac{
\hat{\Psi }_{1}(\omega )}{\delta_{1}\widetilde{\gamma}_{1}},
\end{equation*}
that is
\begin{multline*}
\Psi_{1}(\cdot )=\delta_{1}\widetilde{\gamma}_{1}\left[ \frac{B_{1}(2\cdot +1)}{2}-B
_{1}(2\cdot )+\frac{B_{1}(2\cdot -1)}{2}\right]
\\
=\left[ \psi^{+}_{r_{1}}(\cdot )+r_{1}\psi_{r_{1}}^{+}(\cdot + 1/2)+ r
_{1}\left( \psi_{r_{1}}^{+}(\cdot -1)+r_{1}\psi_{r_{1}}^{+}(\cdot -1/2)\right) \right]
\\
-\left[ \psi^{+}_{1/r_{1}}(\cdot )+r_{1}\psi_{1/r_{1}}^{+}(\cdot - 1/2)+
r_{1}\left( \psi_{1/r_{1}}^{+}(\cdot +1)+r_{1}\psi_{1/r_{1}}^{+}(
\cdot +1/2)\right) \right]
\\
=(S_{1}^{- 1}S_{1}^{+ 1/2}\psi_{r_{1}})(x)-
(R_{1}^{+1}R_{1}^{-1/2}
\psi_{1/r_{1}})(x)
.
\end{multline*}
Combination {\eqref{two-scale}} with substitution $x=y-1/2$ brings, in
particular,
%
\begin{equation}
\label{dym1}\frac{\Psi_{1}(y-1/2)}{\delta_{1}\widetilde{\gamma}_{1}}=
\frac{1}{2}B_{1}(2y)-B
_{1}(2y-1)+\frac{1}{2}B_{1}(2y-2)=B_{1}(y)-2B_{1}(2y-1).
\end{equation}

If $n\ge 2$ then, similarly to the previous step, we pair elements of
$\mathbb{T}^{1}$ by matching two functions $T^{1}_{t_{2},\ldots ,t
_{n}}$ and $T^{1}_{t'_{2},\ldots ,t'_{n}}$ such that $t_{j}=t'_{j}$ for
$j=3,\ldots ,n$, but $t_{2}=1/t'_{2}$. Each couple we associate with the
function
\begin{equation*}
T^{2}_{t_{3},\ldots ,t_{n}}=\frac{1}{\sqrt{r_2}}T^{1}_{r_{2},\ldots ,t_{n}}-\sqrt{r_2}T^{1}_{\frac{1}{r
_{2}},\ldots ,t_{n}}
\end{equation*}
and call $\mathbb{T}^{2}$ the new collection of $2^{n-2}$ functions
$T^{2}_{t_{3},\ldots , t_{n}}$. Again, if $n=2$, we stop the process
with the function $T^{2}$ such that
\begin{equation*}
\frac{\hat{T}^{2}(\omega )}{\delta_{2}\widetilde{\gamma}_{2}}=\frac{\hat{B}_{2}(
\omega /2)}{8}\sum_{k=0}^{3}\frac{3!(-1)^{k}}{k!(3-k)!}\mathrm{e}^{(2-k)i
\omega /2}=
\frac{\hat{B}_{2}(\omega /2)}{2}\left( \frac{\mathrm{e}
^{i\omega }}{4}-\frac{3\mathrm{e}^{i\omega /2}}{4}+ \frac{3}{4}-\frac{
\mathrm{e}^{-i\omega /2}}{4}\right) =:\frac{\hat{\Psi }_{2}(\omega )}{
\delta_{2}\widetilde{\gamma}_{2}},
\end{equation*}
that is
\begin{multline*}
\Psi_{2}(x)=\delta_{2}\widetilde{\gamma}_{2}\Bigl[\frac{B_{2}(2x+2)}{4}-\frac{3B
_{2}(2x+1)}{4}+\frac{3B_{2}(2x)}{4}-\frac{B_{2}(2x-1)}{4}\Bigr]
\\
= \Bigl[\bigl(S_{2}^{\mp 1}S_{2}^{\pm 1/2}S_{1}^{\mp 1} S_{1}^{\pm 1/2}
\psi_{r_{1},r_{2}}^{+}\bigr)(x)-
\bigl(S_{1}^{\mp 1}S_{1}^{\pm 1/2}R
_{1}^{\pm 1} R_{1}^{\mp 1/2}\psi_{1/r_{1},r_{2}}^{+}\bigr)(x)\Bigr]
\\
-\Bigl[\bigl(S_{1}^{\mp 1}S_{1}^{\pm 1/2}R_{1}^{\pm 1} R_{1}^{\mp 1/2}
\psi_{r_{1},1/r_{2}}^{+}\bigr)(x)-
\bigl(R_{2}^{\pm 1}R_{2}^{\mp 1/2}R
_{1}^{\pm 1} R_{1}^{\mp 1/2}\psi_{1/r_{1},1/r_{2}}^{+}\bigr)(x)\Bigr]
.
\end{multline*}
In view of {\eqref{two-scale}},
%
\begin{multline}
\label{dym2}\frac{\Psi_{2}(x)}{\delta_{2}\widetilde{\gamma}_{2}}=\frac{1}{4}\Bigl[B_{2}(2x+2)-3B
_{2}(2x+1)+3B_{2}(2x)-B_{2}(2x-1)\Bigr]=B_{2}(x+1)-\frac{3}{2}B_{2}(2x+1)-
\frac{1}{2}B_{2}(2x-1).
\end{multline}
If $n\ge 3$ we continue the process. At $j$-th step we deal with
functions
\begin{equation*}
T^{j-1}_{t_{j},\ldots ,t_{n}}=\frac{1}{\sqrt{r_{j-1}}}\,T^{j-2}_{r_{j-1},t_{j},\ldots , t_{n}}-\sqrt{r_{j-1}}\,T
^{j-2}_{1/r_{j-1},t_{j},\ldots , t_{n}}
\end{equation*}
from $\mathbb{T}^{j-1}$ and form the new collection $\mathbb{T}^{j}$ of
$2^{n-j}$ elements
%
\begin{equation}
\label{Tj}T^{j}_{t_{j+1},\ldots ,t_{n}}=\frac{1}{\sqrt{r_j}}T^{j-1}_{r_{j},t_{j+1},\ldots ,t_{n}}-\sqrt{r_j}T
^{j-1}_{1/r_{j},t_{j+1},\ldots ,t_{n}}.
\end{equation}
Overall, starting from {\eqref{lambda}} requiring $2n$ steps, and making
exactly $n$ steps of the form {\eqref{Tj}}, that is $3n$ steps in total,
we obtain the localised function $T^{n}$ with
\begin{equation*}
\frac{\hat{T}^{n}(\omega )}{\delta_{n}\widetilde{\gamma}_{n}}=\frac{\hat{B}_{n}(
\omega /2)}{2^{n+1}}\sum_{k=0}^{n+1}\frac{(-1)^{k}(n+1)!}{k!(n+1-k)!}
\mathrm{e}^{(n-k)i\omega /2}=:\frac{\hat{\Psi }_{n}(\omega )}{\delta_{n}\widetilde{\gamma}_{n}}.
\end{equation*}
From here {\eqref{psi_fint}} follows by {\eqref{ddiff}}.\end{proof}
 
\begin{remark} We have ${\rm supp}\,\Psi_n=[-n/2,n/2+1].$ By \eqref{two-scale}, it holds, where $[p]$ stands for integer part of $p$:
\begin{multline}
\label{dymm}\frac{\hat{\Psi }_{n}(\omega )}{\widetilde{\gamma}_{n}\delta_{n}\mathrm{e}^{in
\omega /2}}=
\frac{\hat{B}_{n}(\omega /2)}{2^{n+1}}\sum_{k=0}^{n+1}
\frac{(-1)^{k}(n+1)!}{k!(n+1-k)!}\mathrm{e}^{-ki\omega /2}
\\
=\hat{B}_{n}(\omega )-
\frac{\hat{B}_{n}(\omega /2)}{2^{n}}\sum_{k=0}
^{[n/2]}\frac{(n+1)!}{(2k+1)!(n-2k)!}\mathrm{e}^{-(2k+1)i\omega /2}.
\end{multline} \end{remark}

$\Phi_n$ and $\Psi_n$ realize {\it the localisation property} of Battle--Lemari\'{e} wavelet systems. 
$\Phi_n$ is constructed by integer shifts of $\phi_n$, which generate the same multiresolution analysis in $L^2(\mathbb{R})$. For $\Psi_n$ we group proper shifts of wavelets $\psi_n=\psi_{t_1,\ldots,t_n}^\pm$ from the systems \eqref{ws}. They constitute bases in subspaces $W_0$ related to multiresolution analysis generated by $B_n(\cdot)$ and $B_n(\cdot\pm \frac{1}{2})$. This localisation property is crucial, in particular, for the estimate $\|f\|_{B^{s}_{pq}(\mathbb{R})}^\circledast\lesssim\|f\|_{B^{s}_{pq}(\mathbb{R})}^\ast$ in the proof of Proposition \ref{prop} in \S~4.

\section{Equivalent norms theorem} \subsection{Prerequisites}
\subsubsection{Nikolskii--Besov spaces} Let $s\in\mathbb{R}$ and $0<p,q\le\infty$. For the definition of Nikolskii--Besov spaces $B_{pq}^s(\mathbb{R})$ (or Besov type spaces, in other terminology, which is more commonly used) we refer to \cite{Tr1, Tr2}. 
If, in addition, $s>\max\{0,1/p-1\}$ then one can define $B_{pq}^s(\mathbb{R})$ as follows.
Let $L^p(\mathbb{R})$ with $0<p\le\infty$ be the set of all Lebesgue measurable functions $f$ on $\mathbb{R}$, quasi--normed by $\|f\|_{L^p(\mathbb{R})}:=\left(\int_{\mathbb{R}}|f(x)|^p\,dx\right)^{1/p}$ with the obvious modification for $p=\infty$.
For $M\in\mathbb{N}$ and $f\in L^p(\mathbb{R})$ put (see e.g. \cite[Remark 2.6]{KLSS}) $(\Delta_h^1f)(\cdot):=f(\cdot+h)-f(\cdot)$, $(\Delta_h^Mf):=\Delta_h^1(\Delta_h^{M-1}f)$ and $$\omega_M(f,t)_p:=\sup_{|h|<t}\|\Delta_h^M f\|_{L^p(\mathbb{R})},\qquad t>0.$$ If $\max\{0,1/p-1\}<s<M$ and $1\le p,q\le\infty$ then $f\in B_{pq}^s(\mathbb{R})$ if and only if $f\in L^p(\mathbb{R})$ and $$\left(\int_0^1\left[t^{-s}\omega_M(f,t)_p\right]^q\frac{dt}{t}\right)^{1/q}<\infty.$$ The theory  and properties of Nikolskii--Besov spaces $B_{pq}^s(\mathbb{R})$ may be found in \cite{EdTr, HTr1, HTr1',HTr2,ST}.

\subsubsection{Sequence spaces} Let $0<p\le\infty$, $0<q\le\infty$ and $s\in\mathbb{R}$. The sequence space $b_{pq}^s$ consists of all sequences
$$\mu=\{\mu_{d\tau}\in\mathbb{C}\colon d\in\mathbb{N}_{-1},\, \tau\in\mathbb{Z}\}$$  such that the quasi--norm
\begin{equation}\label{seqnorm}\|\mu\|_{b_{pq}^s}:=\Biggl(\sum_{d=-1}^\infty \biggl(\sum_{\tau\in\mathbb{Z}} |\mu_{d\tau}|^p\biggr)^{\frac{q}{p}}\Biggr)^{\frac{1}{q}}<\infty\end{equation} (wth the usual modification if $p=\infty$ and/or $q=\infty$) is finite. Sequence spaces of the type $b_{pq}^s$ were introduced in \cite{24,25} in connection with atomic decomposition of the spaces $B_{pq}^s(\mathbb{R})$. 

\subsubsection{Spline bases in $B_{pq}^s(\mathbb{R})$} Let $\mathscr{S}(\mathbb{R})$ be the Schwartz space of all complex--valued rapidly decreasing, infinitely differentiable functions on $\mathbb{R}$, and let $\mathscr{S}'(\mathbb{R})$ denote its dual space of tempered distributions.
\begin{proposition}\cite{Tr5}\label{prop4} Let $n\in\mathbb{N}$ and let $$\Bigl\{h_{d\tau}^n\colon d\in\mathbb{N}_{-1},\ \tau\in\mathbb{Z}\Bigr\}$$ be an ($L^\infty$--normalised) orthogonal spline basis in $L^2(\mathbb{R})$ according to \eqref{system}.

Let $0<p\le\infty$, $0<q\le\infty$ and $$\max\biggl\{\frac{1}{p},1\biggr\}-1-n<s<n+\min\biggl\{\frac{1}{p},1\biggr\}.$$ Let $f\in\mathscr{S}'(\mathbb{R})$. Then $f\in B_{pq}^{s}(\mathbb{R})$ if and only if it can be represented as
\begin{equation}\label{function}
f=\sum_{d\in\mathbb{N}_{-1}}\sum_{\tau\in\mathbb{Z}}\mu_{d\tau}2^{-d(s-\frac{1}{p})}h_{d\tau}^n,\qquad \mu\in b_{pq}^s,
\end{equation} unconditional convergence being in $\mathscr{S}'(\mathbb{R})$ and locally in any space $B_{pq}^\sigma(\mathbb{R})$ with $\sigma<s$. The representation \eqref{function} is unique, $$\mu_{d\tau}=\mu_{d\tau}(f)=2^{d(s-\frac{1}{p}+1)}\int_{\mathbb{R}} f(x) h_{d\tau}^n(x)\,dx,\qquad d\in\mathbb{N}_{-1},\ \tau\in\mathbb{Z},$$ and $$J\colon f\mapsto \mu(f)$$ is an isomorphic map of $B_{pq}^s(\mathbb{R})$ onto $b_{pq}^s$. If, in addition, $p<\infty$, $q<\infty$ then $$\Bigl\{2^{-d(s-\frac{1}{p})}h_{d\tau}^n\colon d\in\mathbb{N}_{-1},\ \tau\in\mathbb{Z}\Bigr\}$$ is an unconditional (normalised) basis in $B_{pq}^s(\mathbb{R})$.
\end{proposition}  A proof of this proposition may be found in \cite[Theorems 2.46, 2.49]{Tr5}.

\subsection{The result} It follows from Proposition \ref{prop4} and \eqref{seqnorm} that one can use  
\begin{equation}\label{seq bpq}
\|f\|_{B^{s}_{pq}(\mathbb{R})}^\ast:=\left(\sum_{\tau\in\mathbb{Z}}\left|\langle f,h^n_{-1,\tau}\rangle\right|^p \right)^{\frac{1}{p}}+\left(\sum_{d=0}^\infty 2^{d(s-\frac{1}{p}+1)q}\left(\sum_{\tau\in\mathbb{Z}}\left| \langle f,h^n_{d\tau}\rangle\right|^p\right)^{\frac{q}{p}}\right)^{\frac{1}{q}}
\end{equation} (with usual modifications for $p=\infty$ and $q=\infty$) as an equivalent characterization of the norm in $B_{pq}^s(\mathbb{R})$, when $0<p\le\infty$, $0<q\le\infty$ and $\max\left\{1/p,1\right\}-1-n<s<n+\min\left\{1/p,1\right\}$. In this part of the paper we establish an equivalent characteristic for \eqref{seq bpq}. 
\begin{proposition}\label{prop} Let $n\in\mathbb{N}$ and let $$\Bigl\{h_{d\tau}^n\colon d\in\mathbb{N}_{-1},\ \tau\in\mathbb{Z}\Bigr\}$$ be an ($L^\infty$--normalised) orthogonal spline basis in $L^2(\mathbb{R})$ according to \eqref{system}.

Let $0<p\le\infty$, $0<q\le\infty$ and $$\max\biggl\{\frac{1}{p},1\biggr\}-1-n<s<n+\min\biggl\{\frac{1}{p},1\biggr\}.$$ Then a distribution $f\in\mathscr{S}'(\mathbb{R})$ belongs to $B_{pq}^{s}(\mathbb{R})$ if and only if \begin{equation}\label{seq}
\|f\|_{B^{s}_{pq}(\mathbb{R})}^\circledast:=\left(\sum_{\tau\in\mathbb{Z}}\left|\langle f,B_{n;\,0,\tau}\rangle\right|^p \right)^{\frac{1}{p}}+\left(\sum_{d=0}^\infty 2^{d(s-\frac{1}{p}+1)q}\left(\sum_{\tau\in\mathbb{Z}}\left| \langle f,B_{2n+1;\,d+1,\tau}^{(n+1)}\rangle\right|^p\right)^{\frac{q}{p}}\right)^{\frac{1}{q}}<\infty
\end{equation} (with usual modifications for $p=\infty$ and $q=\infty$). Furthermore, $\|f\|_{B^{s}_{pq}(\mathbb{R})}^\circledast$ may be used as an equivalent norm on $B_{pq}^s(\mathbb{R})$.
\end{proposition}

\begin{proof} 
Let $\phi_n=\phi_{n}^+$ with $c_r=n$. Argumentation for other cases of $\phi_n$ is analogous.
We write, by 
\eqref{phin},
\begin{multline*}
\langle f,\phi_n^+(\cdot -\tau)\rangle=\beta_n\sum_{l_1=0}^\infty (-r_1)^{l_1}\ldots \sum_{l_n=0}^\infty (-r_n)^{l_n}\langle f,B_n(\cdot -\tau+ l_1\ldots + l_n)\rangle\\
=\beta_n\sum_{l_1\le 0} (-r_1)^{-l_1}\ldots \sum_{l_n\le 0} (-r_n)^{-l_n}\langle f,B_n(\cdot -\tau- l_1\ldots - l_n)\rangle.
\end{multline*} 
Then, if $p\le 1$,
\begin{multline}\label{hvost<1}
\sum_{\tau\in\mathbb{Z}}\left|\langle f,\phi_n^+(\cdot -\tau)\rangle\right|^p\le\beta_n^p\sum_{\tau\in\mathbb{Z}}
\sum_{l_1\le 0} r_1^{-l_1p}\ldots \sum_{l_n\le 0} r_n^{-l_np}|\langle f,B_n(\cdot -\tau- l_1\ldots -l_n)\rangle|^p\\=\beta_n^p\sum_{\tau\in\mathbb{Z}}\sum_{l_1\le 0} r_1^{-l_1p}\ldots \sum_{l_{n-1}\le 0} r_{n-1}^{-l_{n-1}p}
\sum_{m\le \tau+l_1+\ldots +l_{n-1}} r_n^{(\tau-m+l_1\ldots +\l_{n-1})p}
|\langle f,B_n(\cdot -m)\rangle|^p\\=\beta_n^p\sum_{m\in\mathbb{Z}}
|\langle f,B_n(\cdot -m)\rangle|^p
\sum_{l_1\le 0} r_1^{-l_1p}\ldots \sum_{l_{n-1}\le 0} r_{n-1}^{-l_{n-1}p}\sum_{\tau\ge m-l_1-\ldots -l_{n-1}} r_n^{(\tau-m+l_1\ldots +\l_{n-1})p}\\\le\beta_n^p
\prod_{j=1}^n\frac{1}{1-r_j^p}\sum_{m\in\mathbb{Z}}
|\langle f,B_n(\cdot -m)\rangle|^p
.\end{multline} Analogously, for $p>1$ we obtain using $n-$times H\"{o}lder's inequality with $p$ and $p'=\frac{p}{p-1}$, and representing $r_j^{l_j}=r_j^{l_j/p+l_j/p'}$ for each $j=1,\ldots,n$, that
\begin{equation}\label{hvost>1}
\biggl(\sum_{k\in\mathbb{Z}}\left|\langle f,\phi_n^+(\cdot -k)\rangle\right|^p\biggr)^{\frac{1}{p}}\le\beta_n
\prod_{j=1}^n\frac{1}{1-r_j}\biggl(\sum_{m\in\mathbb{Z}}
|\langle f,B_n(\cdot -m)\rangle|^p\biggr)^{\frac{1}{p}}
.\end{equation} On the other side, according to the construction of $\Phi_n$ (see \eqref{phi_fint}) with $c_r=n$ and $c_{1/r}=0$ in our case,
\begin{multline*}
\beta_n\biggl(\sum_{\tau\in\mathbb{Z}}
|\langle f,B_n(\cdot -\tau)\rangle|^p\biggr)^{\frac{1}{p}}=
\biggl(\sum_{\tau\in\mathbb{Z}}
|\langle f,\Phi_n(\cdot -\tau)\rangle|^p\biggr)^{\frac{1}{p}}\\\le
 \biggl(\sum_{\tau\in\mathbb{Z}}
|\langle f,\phi_n^+(\cdot -\tau)\rangle|^p\biggr)^{\frac{1}{p}}
\begin{cases}\Bigl(\prod_{j=1}^n(1+r_j^p)\Bigr)^{\frac{1}{p}}, & p\le 1\\
\prod_{j=1}^n(1+r_j), & p>1\end{cases}.
\end{multline*} Thus, \begin{multline}\label{eqvPhi}
\beta_n C_1\biggl(\sum_{\tau\in\mathbb{Z}}
|\langle f,B_n(\cdot -\tau)\rangle|^p\biggr)^{\frac{1}{p}}
\le
\biggl(\sum_{\tau\in\mathbb{Z}}
|\langle f,\phi_n^+(\cdot -\tau)\rangle|^p\biggr)^{\frac{1}{p}}\\
\le\beta_n\biggl(\sum_{\tau\in\mathbb{Z}}
|\langle f,B_n(\cdot -\tau)\rangle|^p\biggr)^{\frac{1}{p}}\begin{cases}
\Bigl(\prod_{j=1}^n\frac{1}{1-r_j^p}\Bigr)^{\frac{1}{p}}, & p\le 1\\
\prod_{j=1}^n\frac{1}{1-r_j}, & p>1\end{cases},
\end{multline} where $C_1^{-1}:=\Bigl(\prod_{j=1}^n(1+r_j^p)\Bigr)^{\frac{1}{p}}$ for $p\le 1$ and $C_1^{-1}:=\prod_{j=1}^n(1+r_j)$ if $p>1$. 

For deriving an estimate similar to \eqref{hvost<1} and \eqref{hvost>1}, but with, say, $\psi_{r_1,\ldots,r_n}^+$ this time, we write, by \eqref{even}, taking into account \eqref{psihat}, \eqref{ti} and \eqref{ess}:
\begin{multline*}{\gamma_n} ^{-1}
\psi_n^+(2^d\cdot -\tau)=\sum_{m_1\le 0}(-r_1)^{-m_1}\sum_{l_1\le 0} (-r_1)^{-l_1}\ldots \sum_{m_n\le 0}(-r_n)^{-m_n}\sum_{l_n\le 0}(-r_n)^{-l_n}\\\Biggl[\sum_{k=0}^{n+1}\frac{(-1)^{k}(n+1)!}{k!(n+1-k)!}\ 
B_n(2^{d+1}\cdot -2\tau-k-2m_1+ l_1-\ldots -2m_n+ l_n)\\
-\biggl\{\!\sum_{j_1\in\{1,\ldots,n\}}\frac{1}{r_{j_1}}\!\biggr\}
\sum_{k=0}^{n+1}\frac{(-1)^{k}(n+1)!}{k!(n+1-k)!}\ 
B_n(2^{d+1}\cdot -2\tau-k+1-2m_1+ l_1-\ldots -2m_n+ l_n)\\
+\!\biggl\{\!\sum_{\substack{j_1,j_2\in\{1,\ldots,n\}\\ j_1\not=j_2}}\!\!\!\frac{1}{r_{j_1}r_{j_2}}\!\biggr\}
\sum_{k=0}^{n+1}\frac{(-1)^{k}(n+1)!}{k!(n+1-k)!}\ 
B_n(2^{d+1}\cdot -2\tau-k+2-2m_1+ l_1-\ldots -2m_n+ l_n)\\+\ldots +\biggl\{\!\!\sum_{\substack{j_1,\ldots,j_c\in\{1,\ldots,n\}\\ j_1\not=\ldots\not=j_c}}\!\!\!\frac{(-1)^c}{r_{j_1}\ldots r_{j_c}}\!\biggr\}
\sum_{k=0}^{n+1}\frac{(-1)^{k}(n+1)!}{k!(n+1-k)!}\ 
B_n(2^{d+1}\cdot -2\tau-k+c-2m_1+ l_1\ldots -2m_n+ l_n)\\\ldots +\frac{(-1)^n}{r_{1}\ldots r_{n}}\ 
\sum_{k=0}^{n+1}\frac{(-1)^{k}(n+1)!}{k!(n+1-k)!}\ 
B_n(2^{d+1}\cdot -2\tau-k+n-2m_1+l_1-\ldots -2m_n+ l_n)\Biggr]
.\end{multline*}
Denote $\mathbb{B}_n(2^{d+1}\cdot-2\tau-2m_1+l_1-\ldots-2m_n+l_n)$ the quantity in the square brackets. Similarly to \eqref{hvost<1},
\begin{multline*}
{|\gamma_n|^{-1}}\sum_{\tau\in\mathbb{Z}}|\langle f,\psi_{r_1,\ldots,r_n}^+(2^d\cdot-\tau)\rangle|^p\\
\le\sum_{\tau\in\mathbb{Z}}\sum_{l_1\le 0} r_1^{-l_1p}\!\!\sum_{m_1\ge 0}r_1^{m_1p}\ldots\sum_{l_n\le 0}r_n^{-l_np} \sum_{m_n\ge 0}r_n^{m_np}|\langle f, \mathbb{B}_n(2^{d+1}\cdot-2\tau-2m_1-l_1-\ldots-2m_n-l_n)\rangle|^p\\=
\sum_{\tau\in\mathbb{Z}}\!\sum_{m_1\ge 0}r_1^{m_1p}\!\!\ldots\!\!\sum_{m_n\ge 0}r_n^{m_np}\!\!
\sum_{\varkappa_1\le 2m_1} r_1^{(2m_1-\varkappa_1)p}\ldots\!\!\!\sum_{\varkappa_n\le 2m_n} r_1^{(2m_n-\varkappa_n)_p} |\langle f, \mathbb{B}_n(2^{d+1}\cdot-2\tau-\varkappa_1-\ldots-\varkappa_n)\rangle|^p\\=
\sum_{\tau\in\mathbb{Z}}\sum_{m_1\ge 0}r_1^{m_1p}\ldots\sum_{m_n\ge 0}r_n^{m_np}
\sum_{\varkappa_1\le 2m_1} r_1^{(2m_1-\varkappa_1)p}\ldots\sum_{\varkappa_{n-1}\le 2m_{n-1}} r_1^{(2m_{n-1}-\varkappa_{n-1})p}\\\sum_{\varkappa\le 2m_{n}+2\tau+\varkappa_1+\ldots+\varkappa_{n-1}} r_n^{(2m_n+2\tau-\varkappa+\varkappa_1+\ldots+\varkappa_{n_1})p}|\langle f, \mathbb{B}_n(2^{d+1}\cdot-\varkappa)\rangle|^p\\=
\sum_{\varkappa\in\mathbb{Z}}|\langle f, \mathbb{B}_n(2^{d+1}\cdot-\varkappa)\rangle|^p\sum_{m_1\ge 0}r_1^{m_1p}\ldots\sum_{m_n\ge 0}r_n^{m_np}
\sum_{\varkappa_1\le 2m_1} r_1^{(2m_1-\varkappa_1)p}\ldots\sum_{\varkappa_{n-1}\le 2m_{n-1}} r_1^{(2m_{n-1}-\varkappa_{n-1})p}\\\sum_{2\tau\ge\varkappa- 2m_{n}-\varkappa_1-\ldots-\varkappa_{n-1}} 
r_n^{(2m_n+2\tau-\varkappa+\varkappa_1+\ldots+\varkappa_{n-1})p}
\le \prod_{j=1}^n\frac{1}{(1-r_j^p)^2}\sum_{\varkappa\in\mathbb{Z}}|\langle f, \mathbb{B}_n(2^{d+1}\cdot-\varkappa)\rangle|^p.\end{multline*} Analogously, for $p>1$
\begin{equation*}
{|\gamma_n|^{-1}}\biggl(\sum_{\tau\in\mathbb{Z}}|\langle f,\psi^+_{r_1,\ldots,r_n}(2^d\cdot-\tau)\rangle|^p\biggr)^{\frac{1}{p}}
\le \prod_{j=1}^n\frac{1}{(1-r_j)^2}\biggl(\sum_{\varkappa\in\mathbb{Z}}|\langle f, \mathbb{B}_n(2^{d+1}\cdot-\varkappa)\rangle|^p\biggr)^{\frac{1}{p}}.\end{equation*}
Since for $p\le 1$ \begin{multline*}
\biggl(\sum_{\varkappa\in\mathbb{Z}}|\langle f, \mathbb{B}_n(2^{d+1}\cdot-\varkappa)\rangle|^p\biggr)^{\frac{1}{p}}\le\Biggl(\sum_{\varkappa\in\mathbb{Z}}\biggl|\sum_{k=0}^{n+1}\frac{(-1)^{k}(n+1)!}{k!(n+1-k)!}\ \langle f, 
B_n(2^{d+1}\cdot -\varkappa-k)\rangle\biggr|^p\\+\biggl\{\sum_{j_1\in\{1,\ldots,n\}}\frac{1}{r_{j_1}}\!\biggr\}^p
\sum_{\varkappa\in\mathbb{Z}}\biggl|\sum_{k=0}^{n+1}\frac{(-1)^{k}(n+1)!}{k!(n+1-k)!}\ \langle f, 
B_n(2^{d+1}\cdot -\varkappa-k+1)\rangle\biggr|^p\\
+\biggl\{\sum_{\substack{j_1,j_2\in\{1,\ldots,n\}\\ j_1\not=j_2}}\!\!\!\frac{1}{r_{j_1}r_{j_2}}\!\biggr\}^p
\sum_{\varkappa\in\mathbb{Z}}\biggl|\sum_{k=0}^{n+1}\frac{(-1)^{k}(n+1)!}{k!(n+1-k)!}\ \langle f, 
B_n(2^{d+1}\cdot -\varkappa-k+2)\rangle\biggr|^p\\\ldots +
\biggl\{\sum_{\substack{j_1,\ldots,j_c\in\{1,\ldots,n\}\\ j_1\not=\ldots\not=j_c}}\!\!\frac{1}{r_{j_1}\ldots r_{j_c}}\!\biggr\}^p
\sum_{\varkappa\in\mathbb{Z}}\biggl|\sum_{k=0}^{n+1}\frac{(-1)^{k}(n+1)!}{k!(n+1-k)!}\ \langle f, 
B_n(2^{d+1}\cdot -\varkappa-k+c)\rangle\biggr|^p\\\ldots +
\biggr\{\frac{1}{r_{1}\ldots r_{n}}\biggr\}^p
\sum_{\varkappa\in\mathbb{Z}}\biggl|\sum_{k=0}^{n+1}\frac{(-1)^{k}(n+1)!}{k!(n+1-k)!}\ \langle f, 
B_n(2^{d+1}\cdot -\varkappa+n-k)\rangle\biggr|^p
\Biggr)^{\frac{1}{p}}\\\le 
\Biggl(1+\sum_{k=1}^n\biggl\{\sum_{\substack{j_1,\ldots,j_k\in\{1,\ldots,n\}\\ j_1\not=\ldots\not= j_k}}\frac{1}{r_{j_1}\ldots r_{j_k}}\biggr\}^p\Biggr)^{\frac{1}{p}}
\Biggl(
\sum_{\varkappa\in\mathbb{Z}}\biggl|\sum_{k=0}^{n+1}\frac{(-1)^{k}(n+1)!}{k!(n+1-k)!}\ \langle f, 
B_n(2^{d+1}\cdot -\varkappa+n-k)\rangle\biggr|^p
\Biggr)^{\frac{1}{p}},
\end{multline*} and, similarly, for $p>1$
\begin{multline*}
\biggl(\sum_{\varkappa\in\mathbb{Z}}|\langle f, \mathbb{B}_n(2^{d+1}\cdot-\varkappa)\rangle|^p\biggr)^{\frac{1}{p}}\\\le 
\biggl(1+\sum_{k=1}^n\sum_{\substack{j_1,\ldots,j_k\in\{1,\ldots,n\}\\ j_1\not=\ldots\not= j_k}}\frac{1}{r_{j_1}\ldots r_{j_k}}\biggr)
\Biggl(
\sum_{\varkappa\in\mathbb{Z}}\biggl|\sum_{k=0}^{n+1}\frac{(-1)^{k}(n+1)!}{k!(n+1-k)!}\ \langle f, 
B_n(2^{d+1}\cdot -\varkappa+n-k)\rangle\biggr|^p
\Biggr)^{\frac{1}{p}},
\end{multline*} we obtain, taking into account \eqref{ess} and \eqref{ddiff},
\begin{multline*}
\biggl(\sum_{\tau\in\mathbb{Z}}|\langle f,\psi^+_{r_1,\ldots,r_n}(2^d\cdot-\tau)\rangle|^p\biggr)^{\frac{1}{p}}\le
\frac{|\gamma_n|}{2^{n+1}}\left(\sum_{\tau\in\mathbb{Z}}\left| \langle f,B_{2n+1;\, d+1,\tau}^{(n+1)}\rangle\right|^p\right)^{\frac{1}{p}}\\\times
\begin{cases}
\Bigl(\prod_{j=1}^n\frac{1}{(1-r_j^p)^2}\Bigr)^{\frac{1}{p}}
\biggl(1+\sum_{k=1}^n\Bigl\{\sum_{\substack{j_1,\ldots,j_k\in\{1,\ldots,n\}\\ j_1\not=\ldots\not= j_k}}\frac{1}{r_{j_1}\ldots r_{j_k}}\Bigr\}^p\biggr)^{\frac{1}{p}}, & p\le 1\\
\prod_{j=1}^n\frac{1}{(1-r_j)^2}\Bigl(1+\sum_{k=1}^n\sum_{\substack{j_1,\ldots,j_k\in\{1,\ldots,n\}\\ j_1\not=\ldots\not= j_k}}\frac{1}{r_{j_1}\ldots r_{j_k}}\Bigr), & p>1\end{cases}.
\end{multline*} For the reverse estimate we write, by \eqref{psi_fint},
\begin{equation*}
\left(\sum_{\tau\in\mathbb{Z}}\left| \langle f,B_{2n+1;\, d+1,\tau}^{(n+1)}\rangle \right|^p\right)^{\frac{1}{p}}=\left(\sum_{\tau\in\mathbb{Z}}\left| \langle f,B_{2n+1;\, d+1,-n+\tau}^{(n+1)}\rangle\right|^p\right)^{\frac{1}{p}}=\frac{2^{2n+1}}{|\widetilde{\gamma}_n|\delta_n}
\left(\sum_{\tau\in\mathbb{Z}}\left| \langle f,\Psi_n(\cdot-\tau)\rangle\right|^p\right)^{\frac{1}{p}}.\end{equation*} By the construction of $\Psi_n$,
\begin{multline*}
\left(\sum_{\tau\in\mathbb{Z}}\left| \langle f,\Psi_n(\cdot-\tau)\rangle\right|^p\right)^{\frac{1}{p}}\\\le \Biggl[
\biggl(\sum_{\tau\in\mathbb{Z}}|\langle f,\psi_{r_1,\ldots,r_n}^+(2^d\cdot-\tau)\rangle|^p\biggr)^{\frac{1}{p}}+
\biggl(\sum_{\tau\in\mathbb{Z}}|\langle f,\psi^+_{r_1,\ldots,r_n}(2^d\cdot-\tau+1/2)\rangle|^p\biggr)^{\frac{1}{p}}\Biggr]\\\times
\begin{cases}2^{n}\Bigl(\prod_{j=1}^n(1+r_j^p)^2\Bigr)^{\frac{1}{p}}, & p\le 1 \\
2^n\prod_{j=1}^n(1+r_j)^2, & p>1\end{cases}.\end{multline*} The required two--sided estimate for the second term in $\|f\|_{B_{pq}^s(\mathbb{R})}^\circledast$ follows from the inequalities above and the fact that $\|f\|_{B^{s}_{pq}(\mathbb{R})}^\ast$ with $\{h_{-10}^n,h_{00}^n\}$ equal to $\{\phi_n(\cdot\pm \frac{1}{2}),\psi^\pm_{n}(\cdot\pm \frac{1}{2})\}$ or $\{\phi_n(\cdot),\psi^\pm_{n}(\cdot)\}$ are equivalent.\end{proof}

\end{document}